\documentclass{article}

\usepackage{blindtext}
\usepackage{subfiles} 

\usepackage{commonpackage}
\usepackage{package_for_list_of_content}
\usepackage{package_to_delete}
\usepackage{macro}

\numberwithin{equation}{section}
\theoremstyle{plain}
\newtheorem{theorem}{Theorem}[section]
\newtheorem{proposition}[theorem]{Proposition}
\newtheorem{corollary}[theorem]{Corollary}
\newtheorem{lemma}[theorem]{Lemma}
\newtheorem{definition}[theorem]{Definition}
\newtheorem{remark}[theorem]{Remark}

\newtheorem{claim}[theorem]{Claim}

\title{A multi-parameter cinematic curvature}
\author{Mingfeng Chen, Shaoming Guo, Tongou Yang}
\date{}

\begin{document}
\maketitle

\begin{abstract}
    We generalize Sogge's cinematic curvature condition to multi-parameter cases, and prove that the associated maximal operator is bounded on $L^p(\R^2)$ for some $p<\infty$. In particular, we prove a local smoothing conjecture of Zahl \cite{Zah23}. 
\end{abstract}

\setcounter{tocdepth}{1}
\tableofcontents

\section{Introduction}

Let $\varphi: \R\to \R$ be a compactly supported smooth bump function. Consider the maximal operator 
\begin{equation}\label{230209e1_1}
    \sup_{|u|\le 1} \anorm{
    \int_{\R} f(x-\theta, y-u\theta)\varphi(\theta)d\theta
    }.
\end{equation}
It is well known that this maximal operator is unbounded on $L^p(\R^2)$ for every $p<\infty$, due to the existence of Nikodym sets. If we introduce curvatures to \eqref{230209e1_1}, by changing the straight lines $(\theta, u\theta)$ to $(\theta, u\theta^2)$ for instance, then Bourgain \cite{Bou86} (see also Marletta and Ricci \cite{MR98}) proved that \begin{equation}\label{230209e1_2}
    \sup_{|u|\le 1} \anorm{
    \int_{\R} f(x-\theta, y-u\theta^2)\varphi(\theta)d\theta
    }
\end{equation}
is bounded on $L^p(\R^2)$ if and only if $p>2$.  Mockenhaupt, Seeger and Sogge \cite{MSS92} also gave a short proof of the above $L^p$ bounds via introducing local smoothing estimates for linear wave equations. The goal of this paper is to study generalizations of \eqref{230209e1_1} and \eqref{230209e1_2}. Consider the maximal operator 
\begin{equation}\label{230209e1_3}
    \sup_{|u|\le  1} \anorm{
    \int_{\R} f(x-\theta, y-u \theta-\theta^2) \varphi(\theta) d\theta
    }. 
\end{equation}
In the definition of the maximal operator \eqref{230209e1_3}, because of the term $\theta^2$, we see that the example of standard Nikodym sets can be avoided. However, by considering the image of a Nikodym set under the maps
\begin{equation}\label{230710e1_4}
    (x, y)\mapsto (x, x^2+y),
\end{equation}
one can check that the maximal operator \eqref{230209e1_3} also fails to be $L^p$ bounded for every $p<\infty$. \\

Note that the last example relies crucially on the fact that we can complete squares. If we change $\theta^2$ in \eqref{230209e1_3} to $\theta^3$, then we will see that the the example no longer works and the maximal operator 
\begin{equation}
    \sup_{|u|\le 1} 
    \anorm{
    \int_{\R}
    f(x-\theta, y-u\theta-\theta^3)\varphi(\theta)d\theta
    }
\end{equation}
is bounded on $L^p(\R^2)$ for some $p<\infty$.  Indeed, we will show that the two-parameter maximal function
\begin{equation}\label{230527e1_6}
    \sup_{|u_1|\le 1, |u_3|\simeq 1} \anorm{
    \int_{\R} f(x-\theta, y-u_1 \theta-u_3\theta^3) \varphi(\theta) d\theta
    }
\end{equation}
is bounded on $L^p(\R^2)$ for some $p<\infty$. Here $|u_3|\simeq 1$ means $1/C\le |u_3|\le C$ where $C>0$ is an arbitrary real number and the $L^p$ bounds depend on $C$; we need $u_3$ to be away from $0$, in order to avoid the example of Nikodym sets. \\

Let us state our main theorem. Let $d\ge 2$. Let $\bfv=(v_1, \dots, v_{d-1})\in \R^{d-1}$.  Take a smooth function $\gamma(\theta; \bfv): \R\times \R^{d-1}\to \R$ and a smooth bump function $\chi(\theta; \bfv)$ supported near the origin. In this paper, we study $L^p$ bounds of the maximal operator 
\begin{equation}\label{240530e1_7}
    \mathcal{M}_{\gamma, \chi} f(x, y):=\sup_{\bfv\in \R^{d-1}} \big|\mathcal{A}_{\gamma, \chi}f(x, y; \bfv)\big|,
\end{equation}
where 
\begin{equation}\label{240530e1_8}
    \mathcal{A}_{\gamma, \chi}f(x, y; \bfv):=\int_{\R} f(x-\theta, y-\gamma(\theta; \bfv))\chi(\theta; \bfv)d\theta. 
\end{equation}
If it is clear from the context what amplitude function $\chi$ is involved, we often abbreviate $\mc{M}_{\gamma, \chi}$ to $\mathcal{M}_{\gamma}$ and $\mc{A}_{\gamma, \chi}$ to $\mc{A}_{\gamma}$.  Denote 
\begin{equation}\label{tagent}
    \bfT(\theta; \bfv):=
    \pnorm{
    \frac{\partial \gamma(\theta; \bfv)}{\partial \theta}, \dots, 
    \frac{\partial^{d} \gamma(\theta; \bfv)}{\partial \theta^{d}}
    }^T.
\end{equation}
We say that $\gamma$ satisfies a $(d-1)$-parameter cinematic curvature condition at the origin if  
\begin{equation}\label{Y_230330nondegenerate}
    \det\begin{bmatrix}
        \frac{\partial \bfT}{\partial \theta}, \frac{\partial \bfT}{\partial v_1}, \dots, \frac{\partial \bfT}{\partial v_{d-1}}
    \end{bmatrix}\Big|_{\theta=0; \bfv=0}\neq 0.
\end{equation}
For the sake of simplicity, if $\gamma$ satisfies \eqref{Y_230330nondegenerate}, then we often say that it is non-degenerate at the origin. 
\begin{theorem}\label{230329theorem1_1}
    Let $d\ge 3$. Let $\gamma(\theta; \bfv): \R\times \R^{d-1}\to \R$ be a smooth function that satisfies the $(d-1)$-parameter cinematic curvature condition as in \eqref{Y_230330nondegenerate}. Then there exists $p_d>0$ depending only on $d$ such that 
    \begin{equation}\label{230527e1_11}
        \norm{
        \mathcal{M}_{\gamma, \chi} f
        }_{L^p(\R^2)} \lesim_{p, \gamma, \chi} \norm{f}_{L^p(\R^2)},
    \end{equation}
    for every $p>p_d$ and every smooth bump function $\chi(\theta; \bfv)$ that is supported in a sufficiently small neighborhood of the origin. The smallness of the support of $\chi(\theta; \bfv)$ depends only on $\gamma$. 
\end{theorem}


The $d=2$ case of Theorem \ref{230329theorem1_1} is due to Sogge \cite{Sog91}, where he introduced the (one-parameter) cinematic curvature condition and generalized Bourgain's result \cite{Bou86} for the maximal operator \eqref{230209e1_2}. The cases $d=3$ and $d=4$ were considered in the recent work Lee, Lee and Oh \cite{LLO23}. We generalize these results to multi-parameter cases. 

If in Theorem \ref{230329theorem1_1} we take $d=3$ and 
\begin{equation}\label{231102e1_12}
 \gamma(\theta; \bfv)=v_1 \theta+v_2\theta^3,   
\end{equation}
then one can compute directly that $\gamma$ satisfies the two-parameter cinematic curvature condition in \eqref{Y_230330nondegenerate} if and only if $v_2$ is away from $0$. Therefore, the bound we claimed above for the maximal operator \eqref{230527e1_6} is a special case of Theorem \ref{230329theorem1_1}. \\

In the proof of Theorem \ref{230329theorem1_1}, we follow the framework of \cite{GGW22}. Indeed, one may use Theorem \ref{230329theorem1_1} to give an alternative proof to the case $m=1$ (in a general dimension $n$) in \cite{GGW22}. We should also emphasize here that similar frameworks already appeared in earlier works in \cite{BGHS21} and \cite{KLO23}. \\

In the two-parameter case $d=3$, by combining Theorem \ref{230329theorem1_1} with the result of Pramanik, Yang and Zahl \cite{PYZ22}, we will see that $p_3=3$ for maximal operators along ellipses. Let us first define these maximal operators. Let $\epsilon_0>0$ be a small real number. Let $\chi(\theta; a, b): \R\times \R^2\to \R$ be a smooth bump function satisfying 
\begin{enumerate}
    \item[1)] $\supp_2(\chi)$ is contained in a sufficiently small neighborhood of $(1, 1)$, where 
    \begin{equation}
        \supp_2(\chi):=\{(a, b): \text{There exists } \theta \text{ such that } (\theta; a, b)\in \supp(\chi)\};
    \end{equation}
    \item[2)] For every $(a, b)$, it holds that
    \begin{equation}
        \supp(\chi(\ \cdot \ ; a, b))\subset (-(1-\epsilon_0)a, (1-\epsilon_0)a).
    \end{equation}
\end{enumerate}
Let us define 
\begin{equation}
    \mc{M}_{\ellipse}f(x, y):= 
    \sup_{a, b}
    \Big|\int_{\R} f\Big(x-\theta, 
    y-b\sqrt{
    1-\pnorm{
    \frac{\theta}{a}
    }^2
    }
    \Big)  \chi(\theta; a, b)d\theta\Big|.
\end{equation}
We have 
\begin{corollary}\label{230609theorem2_1}
    For every $0<\epsilon_0<1/2$ and every $\chi$ satisfying the above 1) and 2), it holds that 
    \begin{equation}
        \norm{
        \mc{M}_{\ellipse} f
        }_{L^p(\R^2)}
        \lesim_{p, \epsilon_0, \chi}
        \norm{f}_{L^p(\R^2)},
    \end{equation}
    for every $p>3$. 
\end{corollary}

Operators of the form $\mathcal{M}_{\ellipse}$ seem to first appear in Erdogan's work \cite{Erd03}. Lee, Lee and Oh \cite{LLO23} proved that $\mathcal{M}_{\ellipse}$ is bounded on $L^p(\R^2)$ for $p>4$, and proved that it is unbounded if $p\le 3$. \\

Next, we introduce more examples that satisfy the cinematic curvature condition \eqref{Y_230330nondegenerate}. Take $d\ge 3$. For $1\le d'\le d$, 
denote 
\begin{equation}
    \bfu_{-d'}:=(u_1, \dots, u_{d'-1}, u_{d'+1}, \dots, u_d). 
\end{equation}
Define 
\begin{equation}\label{231124e1_18}
    \gamma(\theta; \bfu_{-d'}):=
    \sum_{1\le d''\le d, d''\neq d'}
    u_{d''} \frac{\theta^{d''}}{(d'')!}.
\end{equation}
Via elementary calculations, we see that if we assume that $1\le d'<d$ and that $|u_{d'+1}|\simeq 1$, then $\gamma(\theta; \bfu_{-d'})$ satisfies the $(d-1)$-parameter cinematic curvature condition \eqref{Y_230330nondegenerate}. Therefore, as a consequence of Theorem \ref{230329theorem1_1}, we immediately obtain 
\begin{corollary}\label{231126corollary1_3}
    For every $d\ge 3, 1\le d'< d$, it holds that 
\begin{equation}\label{231124e1_19}
    \Norm{
    \sup_{
    \substack{
    |u_{d''}|\lesim 1, d''\neq d'\\
    |u_{d'+1}|\simeq 1
    }
    }
    \anorm{
    \int_{0}^1 
    f(x-\theta, y-\gamma(\theta; \bfu_{-d'}))d\theta
    }
    }_{
    L^p(\R^2)
    } 
    \lesim_{p, \gamma} \norm{f}_{L^p(\R^2)},
\end{equation}
for some $p<\infty$. Here the notation $d''$ in \eqref{231124e1_19} is the same as the one in \eqref{231124e1_18}. 
\end{corollary}

Moreover, for $1< d'< d$, if we remove the constraint $|u_{d'+1}|\simeq 1$ in \eqref{231124e1_19}, and consider the maximal operator 
\begin{equation}\label{231126e1_20}
    \sup_{
    \substack{
    |u_{d''}|\lesim 1, d''\neq d'
    }
    }
    \anorm{
    \int_{0}^1 
    f(x-\theta, y-\gamma(\theta; \bfu_{-d'}))d\theta
    },
\end{equation}
then it is easy to see that \eqref{231126e1_20} is not bounded on $L^p(\R^2)$ for any $p<\infty$, as it is stronger than the maximal function 
\begin{equation}
    \sup_{|u_1|\le 1} \anorm{
    \int_0^1 
    f(x-\theta, y-u_1 \theta)d\theta
    }.
\end{equation}
It remains as an interesting question whether or not for the case $d'=1$ one can remove the constraint $|u_{d'+1}|\simeq 1$ in \eqref{231124e1_19}, and still prove an analogue of \eqref{231124e1_19}. Let us be more precise. Take $d\ge 3$. Denote 
\begin{equation}
    \bfu_{-1}:=(u_2, u_3, \dots, u_d),
\end{equation}
and 
\begin{equation}
    \gamma(\theta; \bfu_{-1}):= u_2\frac{\theta^2}{2!}+u_3\frac{\theta^3}{3!}+\dots+
    u_d\frac{\theta^d}{d!}.
\end{equation}
One very interesting question that remains is to prove 
    \begin{equation}\label{231127e1_24}
            \Norm{
    \sup_{
    \substack{
    |u_{d''}|\lesim 1, 2\le d''\le d
    }
    }
    \anorm{
    \int_{0}^1 
    f(x-\theta, y-\gamma(\theta; \bfu_{-1}))d\theta
    }
    }_{
    L^p(\R^2)
    } 
    \lesim_{p, \gamma} \norm{f}_{L^p(\R^2)},
    \end{equation}
    for some $p<\infty$. The case $d=3$  is relatively easier, and can be handled via the argument in the current paper. However, when $d$ gets larger, singularities of the multi-parameter cinematic curvatures become much more complicated to study, and our method does not seem to get even close to a good understanding of \eqref{231127e1_24}. \\

Next, let us state a corollary of Corollary \ref{231126corollary1_3}. Recall that Bourgain \cite{Bou86} and Marstrand \cite{Mar87} independently proved that if $E\subset \R^2$ has the following property that 
\begin{equation}
    \mc{L}^2\pnorm{\set{
    (x, y)\in \R^2: (x, y)+ rS^1\subset E \text{ for some } r>0
    }
    }>0,
\end{equation}
where $S^1$ is the unit circle on $\R^2$ and $\mc{L}^2$ denotes the two dimensional Lebesgue measure, then $E$ itself also must have positive Lebesgue measure. In other words, Nikodym sets for circles do not exist. We generalize this result to polynomial curves without linear terms, and obtain that Nikodym sets for polynomial curves without linear terms do not exist. 
\begin{corollary}\label{231127corollary1_4}
    Let $E\subset \R^2$ be a measurable set. For $\lambda>0$, let $E_{\lambda}$ be the collection of points $(x, y)\in \R^2$ such that 
    \begin{equation}
    \mc{L}^1\pnorm{
    \set{
    |\theta|\le 1: (x+\theta, y+ u_2\theta^2+u_3 \theta^3+\dots+ u_d\theta^d)\in E
    }
    }\ge \lambda
    \end{equation}
    for some $u_2, \dots, u_d\in \R$. If $\mc{L}^2(E_{\lambda})>0$ for some $\lambda>0$, then $E$ itself also must have positive measure, that is, $\mc{L}^2(E)>0$. 
\end{corollary}
If we had proven \eqref{231127e1_24} already for some $p<\infty$, then Corollary \ref{231127corollary1_4} would follow immediately, by taking $f$ in \eqref{231127e1_24} to be the indicator function for the set $E$ in Corollary \ref{231127corollary1_4}. However, we will see later that to obtain Corollary \ref{231127corollary1_4}, we do not need \eqref{231127e1_24}, and the simpler version in Corollary \ref{231126corollary1_3} will be sufficient. 

\bigskip

In the end, we generalize Theorem \ref{230329theorem1_1} to the case of variable coefficients. Let $\vec{\Gamma}(x, y; \theta; \bfv)$ be a smooth map from $\R^2\times \R\times \R^{d-1}$ to $\R^2$. Let $\chi: \R^2\times \R\times \R^{d-1}\to \R$ be a smooth bump function supported near the origin. Consider the operator 
\begin{equation}\label{230603ea_1}
    \sup_{\bfv\in \R^{d-1}}
    \Big|\int_{\R} f(
    \vec{
    \Gamma}(x, y; \theta; \bfv)
    )
    \chi(x, y; \theta; \bfv)d\theta\Big|
\end{equation}
Without loss of generality let us assume that $\vec{\Gamma}(0)=0$. Moreover, assume that 
\begin{enumerate}
    \item[(H1)] 
    \begin{equation}
        \det
    \begin{bmatrix}
        \frac{\partial \vec{\Gamma}}{\partial x}(0), \ \ \frac{\partial \vec{\Gamma}}{\partial y}(0)
    \end{bmatrix}\neq 0,
    \end{equation} 
    \item[(H2)] 
    \begin{equation}
        \frac{
        \partial\vec{\Gamma}
        }{\partial \theta}(0)\neq 0,
    \end{equation}
\end{enumerate}
where $\vec{\Gamma}$ is written in the column form. The assumption (H1) says that $\Gamma(x, y; 0; 0)$ is locally a diffeomorphism in $x, y$. The assumption (H2) guarantees that the curve $\Gamma(0; \theta; 0)$ does not degenerate to a single point for instance. These two assumptions are minimal assumptions to make sense of the maximal operator \eqref{230603ea_1}. By linear change of variables in $x, y$, let us assume that 
\begin{equation}
    \begin{bmatrix}
        \frac{\partial \vec{\Gamma}}{\partial x}(0), \ \ \frac{\partial \vec{\Gamma}}{\partial y}(0)
    \end{bmatrix}
    =
    \begin{bmatrix}
        1, & 0\\
        0, & 1
    \end{bmatrix}
\end{equation}
This allows us to write 
\begin{equation}
    \vec{\Gamma}(x, y; \theta; \bfv)
    =
    (
    x+\gamma_1(x, y; \theta; \bfv), y+\gamma_2(x, y; \theta; \bfv)
    ),
\end{equation}
with 
\begin{equation}
    \frac{\partial \gamma_{\iota}}{\partial x}(0)
    =\frac{\partial \gamma_{\iota}}{\partial y}(0)=0, \ \ \ \iota=1, 2.
\end{equation}
Next, by the assumption (H2), and by a shearing transform in $x, y$, we can without loss of generality assume that 
\begin{equation}
    \frac{\partial \gamma_1}{\partial \theta}(0)\neq 0, \ \ \frac{\partial \gamma_2}{\partial \theta}(0)= 0.
\end{equation}
Thus, by a nonlinear change of variable in $\theta$, we without loss of generality assume that 
\begin{equation}\label{230615a_8}
    \vec{\Gamma}(x, y; \theta; \bfv)
    =
    (x-\theta, y-\gamma(x, y; \theta; \bfv)):=(x-\theta,h(x,y;\theta;\bfv)),
\end{equation}
with 
\begin{equation}
    \frac{\partial \gamma}{\partial x}(0)=\frac{\partial \gamma}{\partial y}(0)=\frac{\partial \gamma}{\partial \theta}(0)=0. 
\end{equation}
Define 
\begin{equation}
    \bfT(x, y; \theta; \bfv)
    =
    \pnorm{h(x,y;\theta;\bfv),
    \frac{\partial h(x, y; \theta; \bfv)}{\partial \theta}, \dots, 
    \frac{\partial^{d} h(x, y; \theta; \bfv)}{\partial \theta^{d}}
    }^T.
\end{equation}
Assume that 
\begin{enumerate}
    \item[(H3)]\label{H3} 
    \begin{equation}
       \det\begin{bmatrix}
        \frac{\partial \bfT}{\partial \theta}+\frac{\partial \bfT}{\partial x}, \frac{\partial \bfT}{\partial y}, \frac{\partial \bfT}{\partial v_1}, \dots, \frac{\partial \bfT}{\partial v_{d-1}}
    \end{bmatrix}\Big|_{x=y=0; \theta=0; \bfv=0}\neq 0.
    \end{equation}
    \end{enumerate}
Note that this is exactly the condition in  \cite[Definition 1.1]{Zah23}.  Under the assumptions (H1), (H2) and (H3), one can repeat the proof of Theorem \ref{230329theorem1_1} and show that the maximal operator \eqref{230603ea_1} satisfies the same $L^p$ bounds. 

\begin{theorem}\label{231102theorem1_3}
    Let $ \gamma(x,y;\theta;\bfv)$ be the function given in \eqref{230615a_8}. If we assume that it satisfies the assumptions (H3), then there exists $p_d>0$ depending only on $d$ such that
    \begin{equation}
        \left\|\sup_{\bfv\in \R^{d-1}}
    \Big|\int_{\R} f(x-\theta,y- \gamma(x, y; \theta; \bfv)
    )
    \chi(x, y; \theta; \bfv)d\theta\Big| \right\|_{L^p(\R^2)}\lesssim_{p,\gamma,\chi} \norm{f}_{L^p(\R^2)},
    \end{equation}
    for every $p>p_d$ and every smooth bump function $\chi(x,y;\theta;\bfv)$ that is supported in a sufficiently small neighbourhood of the origin. The smallness of the support of $\chi(x,y;\theta;\bfv)$ depends only on $\gamma$.
\end{theorem}

Note that if $\gamma(x, y; \theta; \bfv)$ is constant in the $x, y$ variables, then the assumption (H3) becomes exactly the $(d-1)$-parameter cinematic curvature condition introduced in \eqref{Y_230330nondegenerate}, and therefore Theorem \ref{230329theorem1_1} is a special case of Theorem \ref{231102theorem1_3}.


    The case $d=2$ in Theorem \ref{231102theorem1_3} is a special case of Sogge \cite{Sog91} and Mockenhaupt, Seeger  and Sogge \cite{MSS93}, where the authors there  obtained local smoothing estimates for general Fourier integral operators. Theorem \ref{231102theorem1_3} is a generalization of these results to the setting of multi-parameters for variable coefficient maximal operators. \\

When proving Theorem \ref{231102theorem1_3}, we will indeed prove something stronger; more precisely, we will prove a ``high frequency decay", that is, 
    \begin{equation}
        \left\|\sup_{\bfv\in \R^{d-1}}
    \Big|\int_{\R} P_k f(x-\theta,y- \gamma(x, y; \theta; \bfv)
    )
    \chi(x, y; \theta; \bfv)d\theta\Big| \right\|_{L^p(\R^2)}\lesssim_{p,\gamma,\chi} 2^{
    -c k
    } \norm{f}_{L^p(\R^2)},
    \end{equation}
for every $p>p_d, k\in \N$, some positive constant $c>0$ depending only on $p$ and the curve $\gamma$. Here $P_k$ is a standard Littlewood-Paley projection operator. In particular, this verifies a local smoothing conjecture of Zahl \cite[page 4]{Zah23}. \\

\noindent {\bf Notation.} 
\begin{enumerate}
    \item We often use boldface letters to refer to vectors. For instance, we often write $\bfx=(x, y)$ and  $\bxi=(\xi, \eta)$. 
    \item We make the convention that $k!=\infty$ whenever $k$ is a negative integer.  
    \item Let $C>1$ and $I\subset \R$ be an interval. We use $C I$ to mean the interval of length $C|I|$ that has the same center as $I$.
    \item For a function $a(x, y)$, we use $\supp(a)$ to denote its support. We will use $x\in \supp(a)$ to mean that there exists some $y$ such that $(x, y)\in \supp(a)$.  We often need to deal with amplitude functions $a(\theta; \bfv; \bxi)$. Denote $\supp_1(a):=\{\theta: \theta\in \supp(a)\}$; similarly we define $\supp_2(a)$ and $\supp_3(a)$. 
    \item Unless otherwise specified, every implicit constant in this paper is allowed to depend on $\gamma$ and $d$, which we suppress from the notation. 
\end{enumerate}

\noindent {\bf Acknowledgements.} S.G. is partly supported by NSF-2044828, and partly by the Nankai Zhide Foundation. T.Y. is partly supported by the American Institute of Mathematics. The authors would like thank Josh Zahl for sharing the curved Nikodym example \eqref{230710e1_4} to them, and  pointing out an error regarding the assumption (H3) in Theorem \ref{231102theorem1_3} in an earlier version of the manuscript. The authors would also like to thank Andreas Seeger for sharing the elliptic maximal operator to them, and thank Josh Zahl for discussing the $L^3(\R^2)$ bounds of this operator during the Oberwolfach workshop ``Incidence Problems in Harmonic Analysis, Geometric Measure Theory, and Ergodic Theory", in June 2023.

\section{More connections to Zahl's work \cite{Zah23}}

The goal of this section is to make the exponent $p_d$ in Theorem \ref{230329theorem1_1} more quantitative, by combining Theorem \ref{230329theorem1_1} with Zahl's results in \cite{Zah23}. More precisely, we will see that in Theorem \ref{230329theorem1_1} it suffices to take 
\begin{equation}
    p_d= d+1. 
\end{equation}
It is not clear to us whether this exponent is sharp or not. \\

Let $d\ge 3$ be an integer. Denote 
\begin{equation}
    \bfv=(v_1, \dots, v_{d-1})\in \R^{d-1}.
\end{equation}
Let $\phi(\theta; x, \bfv)$ be a smooth function defined on $\R\times \R\times \R^{d-1}$. Let $\chi(\theta; x, \bfv)$ be a compactly supported smooth function. Define 
\begin{equation}
\mathfrak{M}_{\delta}f(x):=\sup_{\bfv}
    \frac{1}{\delta}
    \anorm{
    \int_{\R}\int_0^{\delta}
    f(\theta, \phi(\theta; x, \bfv)-y')  \chi(\theta; x, \bfv)d\theta dy'
    }.
\end{equation}
Moreover, let us assume that 
\begin{equation}\label{240530e2_4}
    \det
    \begin{bmatrix}
        \partial_x \phi, & \partial_{\bfv} \phi\\
        \partial_x \partial_{\theta} \phi, & \partial_{\bfv}\partial_{\theta}\phi\\
        \dots, & \dots\\
        \partial_x\partial_{\theta}^{d-1}, & \partial_{\bfv}\partial_{\theta}^{d-1}\phi
    \end{bmatrix} \neq 0
\end{equation}
at every point on the support of the function $\chi$. Then it is proven in Zahl \cite{Zah23} that for every $\epsilon>0$, it holds that 
\begin{equation}\label{240530e2_5}
    \norm{
    \mathfrak{M}_{\delta}f
    }_{L^{d}(\R)} \lesim_{\epsilon, \phi, \chi} \delta^{-\epsilon} \norm{f}_{L^d(\R^2)},
\end{equation}
for every $\delta\in (0, 1)$. \\

Recall the maximal operator in \eqref{240530e1_7} and the averaging operator in \eqref{240530e1_8}. Let us assume that $\gamma$ is a smooth function that satisfies the $(d-1)$-parameter cinematic curvature condition as in \eqref{Y_230330nondegenerate}. Let $P_k$ be a standard Littlewood-Paley projection. We will prove that 
\begin{equation}
    \norm{
    \mc{M}_{\gamma, \chi} P_k f
    }_{
    L^{d+1}(\R^2)
    }
    \lesim_{\epsilon, \gamma, \chi} 2^{\epsilon k}
    \norm{f}_{L^{d+1}(\R^2)},
\end{equation}
for every $\epsilon>0$ and every $k\in \N$. We will see that this follows directly from Zahl's bound \eqref{240530e2_5}. Indeed, by using the bound \eqref{240530e2_5} we can prove something much stronger. More precisely, we can show that 
\begin{equation}\label{240530e2_7}
    \Norm{
    \sup_{\bfv\in \R^{d-1}}\sup_{y\in \R} 
    |
    \mc{A}_{\gamma, \chi} P_k f(x, y; \bfv)
    |
    }_{L^{d+1}_x(\R)} \lesim_{\epsilon, \gamma, \chi} 2^{\epsilon k} \norm{f}_{L^{d+1}(\R^2)}.
\end{equation}
When checking Zahl's curvature condition \eqref{240530e2_4} for the maximal operator \eqref{240530e2_7}, we see exactly our $(d-1)$-parameter cinematic curvature condition as in \eqref{Y_230330nondegenerate}. In the end, we just need to interpolate the bound  \eqref{240530e2_7} with Theorem \ref{230329theorem1_1} (indeed its stronger version Proposition \ref{230329prop2_2}), and finish the proof of the claim at the beginning of this section that one can take $p_d=d+1$ in Theorem \ref{230329theorem1_1}.

\section{Maximal operators along ellipses}

The goal of this section to prove Corollary \ref{230609theorem2_1}, by assuming Theorem \ref{230329theorem1_1}, and by applying an $L^3(\R^2)$ bound due to Pramanik, Yang and Zahl \cite{PYZ22}.

\subsection{An $L^3$ bound}

Let $\delta>0$ be a small number. Define 
\begin{equation}
    \mc{M}_{\ellipse, \delta}f(x, y):= 
    \sup_{a, b}
    \frac{1}{\delta}
    \Big|\int_0^{\delta}\int_{\R} f\Big(x-\theta, 
    y-b\sqrt{
    1-\pnorm{
    \frac{\theta}{a}
    }^2
    }-y'
    \Big)  \chi(\theta; a, b)d\theta dy'\Big|.
\end{equation}
The goal of this section is to prove that 
\begin{equation}\label{230609e2_6}
    \norm{
    \mc{M}_{\ellipse, \delta}f
    }_{L^3(\R^2)}
    \lesim_{\epsilon, \chi} \delta^{-\epsilon}
    \norm{f}_{L^3(\R^2)},
\end{equation}
for every $\epsilon>0$, and every $\delta\in (0, 1)$. Such bounds will follow essentially from \cite{PYZ22}, which we state as follows.  \\

Let $\delta>0$ be a small number. Let $\chi(\theta; a, b)$ be a compactly supported smooth function. Define 
\begin{equation}
    \mathfrak{M}_{\delta}f(x):=\sup_{a, b}
    \frac{1}{\delta}
    \anorm{
    \int_{\R}\int_0^{\delta}
    f(\theta, \phi(\theta; x, a, b)-y')  \chi(\theta; a, b)d\theta dy'
    }
\end{equation}
Moreover, let us assume that $\phi$ is a smooth function satisfying 
\begin{equation}\label{230609e2_17}
\det
\begin{bmatrix}
\partial_x\phi, \ \partial_a\phi, \ \partial_b \phi\\
\partial_x\partial_{\theta}\phi, \ \partial_a\partial_{\theta}\phi, \ \partial_b\partial_{\theta} \phi\\
\partial_x\partial^2_{\theta}\phi, \ \partial_a\partial^2_{\theta}\phi, \ \partial_b\partial^2_{\theta} \phi
\end{bmatrix}
\neq 0
\end{equation}
at every point on the support of $\chi$. Then it is proven\footnote{The result proved in \cite{PYZ22} is indeed a lot stronger; for instance, a key point in the proof of \cite{PYZ22} is that only $C^2$ regularity is needed for the curve $\gamma$. } in \cite{PYZ22} that 
for every $\epsilon>0$, it holds that 
\begin{equation}
    \norm{
    \mathfrak{M}_{\delta} f
    }_{L^3(\R)} 
    \lesim_{\epsilon, \phi, \chi} 
    \delta^{-\epsilon}
    \norm{f}_{L^3(\R^2)},
\end{equation}
for every $\delta\in (0, 1)$. \\

To apply the result of \cite{PYZ22}, 
we break the operator $\mc{M}_{\ellipse, \delta}$ into two parts as follows
\begin{equation}\label{230609e2_10}
    \begin{split}
        & 
        \sup_{a, b}
    \frac{1}{\delta}
    \Big|\int_0^{\delta}\int_{-\epsilon_1}^{\epsilon_1} f\Big(x-\theta, 
    y-b\sqrt{
    1-\pnorm{
    \frac{\theta}{a}
    }^2
    }-y'
    \Big)  \chi(\theta; a, b)d\theta dy'\Big|\\
    & + 
    \sup_{a, b}
    \frac{1}{\delta}
    \Big|\int_0^{\delta}\int_{|\theta|\ge \epsilon_1} f\Big(x-\theta, 
    y-b\sqrt{
    1-\pnorm{
    \frac{\theta}{a}
    }^2
    }-y'
    \Big)  \chi(\theta; a, b)d\theta dy'\Big|,
    \end{split}
\end{equation}
where $\epsilon_1$ is a small positive number that will determined later. To control the contribution from the latter term in \eqref{230609e2_10}, we freeze the $x$-variable, and denote 
\begin{equation}
\phi(\theta; y, a, b):=y-b\sqrt{
    1-\pnorm{
    \frac{\theta}{a}
    }^2
    }.
\end{equation}
We directly check the curvature condition given as in \eqref{230609e2_17} for the function $\phi(\theta;y, a, b)$, and see that the determinant is equal to 
\begin{equation}
\frac{2 b \theta^3}{a (a^2 - \theta^2)^3}.
\end{equation}
This computation explains the decomposition in \eqref{230609e2_10}. It therefore remains to control the contribution from the former term. Recall that in the assumption 1) at the beginning of this section, we assumed that $(a, b)$ takes values in a sufficiently small neighborhood of $(1, 1)$. Moreover, we have the freedom of picking $\epsilon_1$ to be sufficiently small. As a result, we can use a simple localization argument, and only need to prove 
\begin{equation}
\Norm{
\sup_{a, b}
    \frac{1}{\delta}
    \Big|\int_0^{\delta}\int_{-\epsilon_1}^{\epsilon_1} f\Big(x-\theta, 
    y-b\sqrt{
    1-\pnorm{
    \frac{\theta}{a}
    }^2
    }-y'
    \Big)  \chi(\theta; a, b)d\theta dy'\Big|
}_{
L^3(B_{100\epsilon_1})
}
\lesim_{\epsilon, \epsilon_1} \delta^{-\epsilon} \norm{f}_{L^3(\R^2)}, 
\end{equation}
for every $\epsilon>1$ and every $\delta\in (0, 1)$. Here $B_{100\epsilon_1}$ is the ball in $\R^2$ of radius $100\epsilon_1$ centered at the origin.  By the change of variables 
\begin{equation}
x\mapsto x, \ \ y\mapsto x+y,
\end{equation}
it suffices to prove that 
\begin{equation}
\Norm{
\sup_{a, b}
    \frac{1}{\delta}
    \Big|\int_0^{\delta}\int_{-\epsilon_1}^{\epsilon_1} f\Big(x-\theta, 
    x-b\sqrt{
    1-\pnorm{
    \frac{\theta}{a}
    }^2
    }-y'
    \Big)  \chi(\theta; a, b)d\theta dy'\Big|
}_{
L_x^3(B_{200\epsilon_1})
}
\lesim_{\epsilon, \epsilon_1} \delta^{-\epsilon} \norm{f}_{L^3(\R^2)}.
\end{equation}
Note that here on the left hand side we are taking an $L^3$ norm in one variable, and $B_{200\epsilon_1}$ is the interval on $\R$ of radius $200\epsilon_1$ centered at the origin. To apply the result in \cite{PYZ22}, the collection of curves we need to consider becomes 
\begin{equation}
\set{\theta\mapsto 
x-
b\sqrt{
1-\pnorm{
\frac{x-\theta}{a}
}^2
}: 
|x|\le 200\epsilon_1, a, b
}.
\end{equation}
Here $(a, b)$ is sufficiently close to $(1, 1)$.  Denote 
\begin{equation}
\phi(\theta; x, a, b):=x-
b\sqrt{
1-\pnorm{
\frac{x-\theta}{a}
}^2
}.
\end{equation} 
By continuity, we only need to check \eqref{230609e2_17} at $x=\theta=0$, $a=b=1$. Via a direct computation, we see that the determinant in \eqref{230609e2_17} equals $-2$ at this point. This finishes the estimate on the first term in \eqref{230609e2_10}.

\subsection{A local smoothing estimate}

The goal of this subsection is to apply Theorem \ref{230329theorem1_1} (indeed its stronger version Proposition \ref{230329prop2_2}), the estimate \eqref{230609e2_6} and a simple interpolation argument to finish the proof of Corollary \ref{230609theorem2_1}. By the triangle inequality, it suffices to prove that 
\begin{equation}\label{230609e2_18}
\norm{
        \mc{M}_{\ellipse} P_k f
        }_{L^p(\R^2)}
        \lesim_{p}
        2^{-\kappa_p k}
        \norm{f}_{L^p(\R^2)},
\end{equation}
for every $p>3$ and some $\kappa_p>0$ that is allowed to depend on $p$.\footnote{Here to simplify notation, we leave out the dependence on $\chi$ and $\epsilon_0$. } Here $P_k f$ is a Littlewood-Paley projection of $f$. First of all, by \eqref{230609e2_6}, we obtain 
\begin{equation}
\norm{
        \mc{M}_{\ellipse} P_k f
        }_{L^3(\R^2)}
        \lesim_{\epsilon}
        2^{\epsilon k}
        \norm{f}_{L^3(\R^2)},
\end{equation}
for every $\epsilon>0$. Therefore by interpolation, it suffices to prove \eqref{230609e2_18} for sufficiently large $p$. Let $\kappa>0$ be a small number that is to be determined. Let $\chi_{2^{-\kappa k}}: \R\to \R$ be an $L^{\infty}$ normalized smooth bump function adapted to the interval $(-2^{-\kappa k}, 2^{-\kappa k})$ such that $1-\chi_{2^{-\kappa k}}$ is supported away from the origin. We split $\mc{M}_{\ellipse}$ into two parts by 
\begin{equation}
\begin{split}
&  \sup_{a, b}
    \Big|\int_{\R} f\Big(x-\theta, 
    y-b\sqrt{
    1-\pnorm{
    \frac{\theta}{a}
    }^2
    }
    \Big) \chi_{2^{-\kappa k}}(\theta) \chi(\theta; a, b)d\theta\Big|\\
    &+ 
     \sup_{a, b}
    \Big|\int_{\R} f\Big(x-\theta, 
    y-b\sqrt{
    1-\pnorm{
    \frac{\theta}{a}
    }^2
    }
    \Big) (1-\chi_{2^{-\kappa k}}(\theta)) \chi(\theta; a, b)d\theta\Big|
\end{split}
\end{equation}
Let us write them as $\mc{M}'_{\ellipse}$ and $\mc{M}''_{\ellipse}$ separately. For the former term, we have the trivial bound 
\begin{equation}
\norm{
\mc{M}'_{\ellipse} f
}_{L^{\infty}}\lesim 2^{-\kappa k}\norm{f}_{L^{\infty}}. 
\end{equation}
To bound the latter term, we will apply Proposition \ref{230329prop2_2}. Via a direct computation, the determinant in \eqref{Y_230330nondegenerate} is equal to 
\begin{equation}
\frac{
6  b^2 \theta 
}{
(a^2 - \theta^2)^{9/2}
}.
\end{equation}
In absolute values, this is $\gtrsim 2^{-\kappa k}$. We therefore apply Proposition \ref{230329prop2_2}\footnote{Here in order to apply Proposition \ref{230329prop2_2} we also need to use the reduction argument in Section \ref{section_normal_form}.}  and obtain that 
\begin{equation}
\norm{
\mc{M}''_{\ellipse} P_k f
}_{L^p}\lesim 2^{
-\frac{k}{p}
+C \kappa k
}\norm{f}_{L^p},
\end{equation}
where $C$ is a large universal constant. Picking $\kappa$ small enough will finish the proof.

\section{Proof of Corollary \ref{231127corollary1_4}}

In this section, we will prove Corollary \ref{231127corollary1_4}. The proof is via a simple limiting argument. 

Assume $\mc{L}^2(E_{\lambda})>0$ for some $\lambda>0$. For every $(x, y)\in E_{\lambda}$, let 
\begin{equation}
    u_2(x, y), \dots, u_d(x, y)
\end{equation}
be functions such that 
\begin{equation}
    \mc{L}^1\pnorm{
    \set{
    |\theta|\le 1: (x+\theta, y+ u_2(x, y)\theta^2+u_3(x, y)\theta^3+\dots+ u_d(x, y)\theta^d)\in E
    }
    }\ge \lambda.
    \end{equation}
Write $E_{\lambda}$ as a disjoint union 
\begin{equation}
    E_{\lambda}=E_{\lambda, 2}\bigcup E'_{\lambda, 2},
\end{equation}
where 
\begin{equation}
    E_{\lambda, 2}:=\{(x, y)\in E_{\lambda}: u_2(x, y)\neq 0\}.
\end{equation}
We discuss two cases separately. The first case is $\mc{L}^2(E_{\lambda, 2})>0$ and the second case is $\mc{L}^2(E_{\lambda, 2})=0$. If we are in the first case, then by a simple limiting argument, we are able to find $C_2>0$ such that 
\begin{equation}
    \mc{L}^2(E_{\lambda, 2, C_2})>0,
\end{equation}
where 
\begin{equation}
    E_{\lambda, 2, C_2}:=\{(x, y)\in E_{\lambda, 2}: |u_2(x, y)|\ge 1/C_2, |u_{d'}|\le C_2, \forall 2\le d'\le d\}.
\end{equation}
By scaling in the $y$ variable, and by Corollary \ref{231126corollary1_3}, we see that $\mc{L}^2(E)>0$. \\

If we are in the second case $\mc{L}^2(E_{\lambda, 2})=0$, then we write 
\begin{equation}
    E_{\lambda}=E_{\lambda, 3}\bigcup E'_{\lambda, 3},
\end{equation}
where 
\begin{equation}
    E_{\lambda, 3}:=\{(x, y)\in E_{\lambda}: u_3(x, y)\neq 0\}.
\end{equation}
We then argue in the same way as above. \\

We continue this process until we reach the last case where 
\begin{equation}
    u_{d'}(x, y)=0
\end{equation}
for every $2\le d'\le d$ and almost every $(x, y)$. In this case, the desired bound $\mc{L}^2(E)>0$ follows from the boundedness of the one-dimensional Hardy-Littlewood maximal operator.

\section{Normal forms}\label{section_normal_form}

\begin{definition}[Normal forms]
For a smooth function $\gamma(\theta; \bfv): \R\times \R^{d-1}\to \R$ satisfying the $(d-1)$-parameter cinematic curvature condition at the origin, we say that it is of a normal form at the origin if it can be written in either one of the two following forms. 
\begin{enumerate}
    \item[(1)] We say that $\gamma(\theta; \bfv)$ is of Form $(I)$ at the origin if 
    \begin{equation}\label{230401e2_1}  \gamma(\theta;\bfv)=\sum_{j=1}^{d'-2}v_j\frac{\theta^{j}}{j!} +Q(\bfv)\frac{\theta^{d'-1}}{(d'-1)!}+(1+v_{d'-1})\frac{\theta^{d'}}{(d')!}+\sum_{j=d'+1}^{d} (\sigma_j+v_{j-1})\frac{\theta^j}{j!}+P(\theta;\bfv)\frac{\theta^{d+1}}{(d+1)!},
\end{equation}
where $d'\in [2,d]$, $\sigma_{j}$ is a fixed real number for each $j$, and 
$Q(\bfv)$ and $P(\theta;\bfv)$ are smooth function with $Q$ satisfying $Q(\bfv)=O(|\bfv|^2)$. 
\item[(2)] We say that $\gamma(\theta; \bfv)$ is of Form $(II)$ at the origin if 
\begin{equation}\label{230401e2_2} \gamma(\theta;\bfv)=\sum_{j=1}^{d-1}v_j\frac{\theta^{j}}{j!} +Q(\bfv)\frac{\theta^{d}}{d!}+(1+L(\bfv))\frac{\theta^{d+1}}{(d+1)!}+P(\theta;\bfv)\frac{\theta^{d+2}}{(d+2)!},
\end{equation}
where $Q(\bfv), L(\bfv), P(\theta; \bfv)$ are smooth functions satisfying $Q(\bfv)=O(|\bfv|^2), L(\bfv)=O(|\bfv|)$. 
\end{enumerate}

\end{definition}
In Form (I), since the coefficient $Q(\bfv)$ of the term $\theta^{d'-1}$ is of order $O(|\bfv|^2)$, which can be thought of as a perturbation term, we call $d'-1$ the ``missing degree" of $\gamma$. The last term in \eqref{230401e2_1} will also be considered as a perturbation term as its order in the $\theta$ variable is high. Analogously, in Form (II) we will call $d$ the missing degree, and the last term in \eqref{230401e2_2} will also be considered as a perturbation term.  \\

\begin{proposition}\label{230329prop2_2}
Let $P_k$ be a Littlewood-Paley projection on $\R^2$.  Under the same assumption as in Theorem \ref{230329theorem1_1} and under the assumption that $\gamma(\theta; \bfv)$ is of a normal form at the origin, we have that 
\begin{equation}\label{230329e2_1}
    \norm{
    \mc{A}_{\gamma} P_kf
    }_{L^p(\R^2\times \R^{d-1})} \lesim_{p, \gamma, \epsilon} 2^{
    -\frac{dk}{p}+\epsilon k
    }\norm{f}_{L^p(\R^2)},
\end{equation}
holds for every $p>p_d, \epsilon>0$ and $k\ge 1$. 
\end{proposition}

\begin{proof}[Proof of Theorem \ref{230329theorem1_1} by assuming Proposition \ref{230329prop2_2}.]
We first prove below that \eqref{230329e2_1} holds for all non-degenerate $\gamma(\theta;\bfv)$ assuming it holds for all $\gamma$ in normal form. Then by standard Sobolev embedding, this implies the desired maximal operator bound as in Theorem \ref{230329theorem1_1}. In the rest of this section we carry out the reduction from a general non-degenerate $\gamma(\theta; \bfv)$ to a normal form. \\

Since $\gamma(\theta;\bfv)$ is smooth, we may write
\begin{equation}    \gamma(\theta;\bfv)=\gamma_0(\bfv)+\sum_{j=1}^{d+1} \left(\sigma_j+
L_j(\bfv)
\right)\frac{\theta^j}{j!}+P(\theta;\bfv)\frac{\theta^{d+2}}{(d+2)!}+Q(\theta;\bfv)\theta,
\end{equation}
where $\sigma_j\in \R$, $L_j(\bfv)$ is a linear form  in $\bfv$ for each $j$, 
$\gamma_0(\bfv)$, $P(\theta;\bfv)$, $Q(\theta;\bfv)$ are smooth functions with $Q$ satisfying $Q(\theta; \bfv)=O(|\bfv|^2)$. We may assume $\gamma_0(\bfv)\equiv 0$. Indeed, when integrating $\mathcal A_\gamma f$ over $\R^2\times \R^{d-1}$, we may change variables
\begin{equation}
    y-\gamma_0(\bfv)\mapsto y.
\end{equation}
We may also assume $\partial_\theta\gamma(0;0)=0$, that is, $\sigma_1=0$, and this can be achieved via the shear transformation
\begin{equation}
    y-\sigma_1 x\mapsto y.
\end{equation}
Thus we arrive at
\begin{equation}    \gamma(\theta;\bfv)=L_1(\bfv)\theta+\sum_{j=2}^{d+1} \left(\sigma_j+L_j(\bfv)\right)\frac{\theta^j}{j!}+P(\theta;\bfv)\frac{\theta^{d+2}}{(d+2)!}+Q(\theta;\bfv)\theta.
\end{equation}
Note that we have not used the non-degeneracy condition \eqref{Y_230330nondegenerate} so far. \\

Now we consider the first column of the matrix in \eqref{Y_230330nondegenerate}. There must be a smallest $d'\in [2,d+1]$ such that $\partial_\theta^{d'}\gamma(0;0)\ne 0$. That means $\partial_\theta^{j}\gamma(0;0)=\sigma_j= 0$ for every $2\le j\le d'-1$. We consider the cases $d'=2, d+1$ and the cases $2<d'<d+1$ separately. \\

Let us first consider the case $d+1> d'>2$. By linear changes of variables in $\bfv$ and the non-degeneracy assumption on $\gamma$, we can assume that 
\begin{equation}
    L_j(\bfv)=v_j,\ \  \forall 1\le j\le d'-2. 
\end{equation}
The function $\gamma(\theta; \bfv)$ becomes 
\begin{equation}    \gamma(\theta;\bfv)=
\sum_{j=1}^{d'-2} v_j \frac{\theta^j}{j!}+
\sum_{j=d'-1}^{d+1} \left(\sigma_j+L_j(\bfv)\right)\frac{\theta^j}{j!}+P(\theta;\bfv)\frac{\theta^{d+2}}{(d+2)!}+Q(\theta;\bfv)\theta,
\end{equation}
and we can without loss of generality assume that $\sigma_{d'}=1$. Recall that $\sigma_{d'-1}=0$ by assumption. Next, we make the change of variables 
\begin{equation}\label{230401e2_10}
    cL_{d'-1}(\bfv)+\theta\mapsto \theta,
\end{equation}
for some appropriately chosen constant $c$, and we can get rid of the linear form $L_{d'-1}(\bfv)$. The function $\gamma(\theta; \bfv)$ becomes 
\begin{equation}    \gamma(\theta;\bfv)=
\sum_{j=1}^{d'-2} v_j \frac{\theta^j}{j!}+
\sum_{j=d'}^{d+1} \left(\sigma_j+L_j(\bfv)\right)\frac{\theta^j}{j!}+P(\theta;\bfv)\frac{\theta^{d+2}}{(d+2)!}+Q(\theta;\bfv)\theta,
\end{equation}
with $\sigma_{d'}=1$. Now we use the non-degeneracy assumption on $\gamma$ again, and see that by linear change of variables in $\bfv$, we can assume that 
\begin{equation}
    L_j(\bfv)=v_{j-1}, \ \ \forall d\ge j\ge d'. 
\end{equation}
To see that $\gamma$ can be further reduced to the desired form \eqref{230401e2_1}, we just apply non-linear changes of variables in $\bfv$. This finishes the argument for the case $2<d'<d+1$. \\

Next we consider the case $d'=d+1$ and the case $d'=2$. They are similar, and we only consider $d'=d+1$. Note that by assumption $\sigma_j=0$ for every $j\ge d$. Therefore by the assumption that $\gamma$ is non-degenerate, we can make linear changes of variables in $\bfv$ so that 
\begin{equation}
    L_j(\bfv)=v_j, \ \ \forall j\le d-1. 
\end{equation}
We have arrived at 
\begin{equation}
    \gamma(\theta; \bfv)=\sum_{j=1}^{d-1} 
    v_j \frac{\theta^j}{j!}
    + 
    L_d(\bfv) \frac{\theta^d}{d!}
    +
    (1+L_{d+1}(\bfv))\frac{\theta^{d+1}}{(d+1)!} 
    +P(\theta;\bfv)\frac{\theta^{d+2}}{(d+2)!}+Q(\theta;\bfv)\theta.
\end{equation}
By a change of variable similar to \eqref{230401e2_10}, we can assume that $L_d(\bfv)\equiv 0$. This finishes the reduction for the case $d'=d+1$.

\end{proof}

\section{Reduction algorithm}

We need to prove Proposition \ref{230329prop2_2}. Let us assume that the missing degree of $\gamma(\theta; \bfv)$ is $d_0-1$ for some $d_0\in [2, d+1]$. 
To simplify notation, we write 
\begin{equation}
   \bfu=(u_1, \dots, u_d)\in \R^d,  \ \ \circu_{d'}=(u_1, \dots, u_{d'-1}, u_{d'+1}, \dots u_d)\in \R^{d-1},
\end{equation}
for $d'=1, \dots, d$. We will write $\gamma$ of Form (I) and missing degree $d_0-1$ as 
\begin{equation}
    \gamma(\theta; \bfu_{d_0-1})=\sum_{j=1, \dots, d, j\neq d_0-1} 
    (\sigma_j+u_j) \frac{\theta^j}{j!} 
    + Q(\bfu_{d_0-1}) \frac{\theta^{d_0-1}}{(d_0-1)!}+ P(\theta; \bfu_{d_0-1}) \frac{\theta^{d+1}}{(d+1)!},
\end{equation}
where $\sigma_j=0$ for $j=1, \dots, d_0-2$ and $\sigma_{d_0}=1$. Similarly, we will write $\gamma$ of Form (II) and missing degree $d$ as
\begin{equation}
    \gamma(\theta; \bfu_{d})=\sum_{j=1, \dots, d-1} 
    u_j \frac{\theta^j}{j!} 
    + (1+L(\bfu_d)) \frac{\theta^{d+1}}{(d+1)!}+ Q(\bfu_{d}) \frac{\theta^{d}}{d!}+ P(\theta; \bfu_{d}) \frac{\theta^{d+2}}{(d+2)!}.
\end{equation}
By taking Fourier transforms and scaling in frequency variables, what we need to prove becomes 
\begin{equation}\label{230401e3_4}
    \Norm{
    \iint_{\R^2} 
    \widehat{f}(\bxi)
    m(\bfu_{d_0-1}; \bxi) 
    e^{i\bfx\cdot \bxi} d\bxi
    }_{L^p (\R^2\times \R^{d-1})} \lesim_{\epsilon} 2^{-\frac{dk}{p}} 2^{\epsilon k}\norm{f}_{L^p(\R^2)},
\end{equation}
for every $\epsilon>0$, every $f$ with $\supp(\widehat{f})\subset \{(\xi, \eta): |\eta|\simeq 1\}$, where
\begin{equation}
    \bfx=(x, y), \ \ \bxi=(\xi, \eta), 
\end{equation}
the multiplier $m$ is given by 
\begin{equation}
    m(\bfu_{d_0-1}; \bxi):=\int_{\R} 
    e^{i2^k 
    \Phi(\theta; \bfu_{d_0-1}; \bxi)
    }
    \chi(\theta; \bfu_{d_0-1})d\theta,
\end{equation}
and the phase function $\Phi$ is given by 
\begin{equation}
    \Phi(\theta; \bfu_{d_0-1}; \bxi):=\theta \xi+\gamma(\theta; \bfu_{d_0-1})\eta.
\end{equation}
To prove \eqref{230401e3_4}, we will run an algorithm to reduce it to a sum of terms that are easier to handle. In this section, we will describe the algorithm, and in the next section, we will handle all the resulting terms. \\

The inputs of the algorithm: The phase function $\Phi(\theta; \circu_{d_0-1}; \bxi)$, a smooth amplitude function 
\begin{equation}
    \mathfrak{a}^{(0)}(\theta; \circu_{d_0-1}; \bxi):=
    \chi(\theta; \bfu_{d_0-1}) a^{(0)}_3(\bxi),
\end{equation} 
where  $a^{(0)}_3(\bxi)$ is supported on $|\eta|\simeq 1, |\xi|\lesim 1$. Under the new notations, the desired estimate \eqref{230401e3_4} can be written as 
\begin{equation}\label{230209e2_8}
    \Norm{
\iint_{\R^2} \widehat{f}(\bxi)
\Big[
\int_{\R} e^{i2^k \Phi(\theta; \circu_{d_0-1}; \bxi)} \mathfrak{a}^{(0)}(\theta; \circu_{d_0-1}; \bxi)d\theta
\Big] e^{i\bfx\cdot \bxi} d\bxi 
}_{L^p} \lesim  2^{-\frac{d \cdot k}{p}} 2^{\epsilon k} \norm{f}_p,
\end{equation}
Define 
\begin{equation}
    \mathfrak{k}_0:=k, \ \ \mathfrak{m}_0:=d_0-1.
\end{equation}
We will call $\mathfrak{m}_0$ the missing degree of the phase function $\Phi$.

\subsection{The first step of the algorithm}\label{230210sub2_1}

Let us describe the first iteration of the algorithm. We cut the $\bfu_{\mathfrak{m}_0}$-support of $\mathfrak{a}^{(0)}$ into $O_d(1)$ many pieces $\{\supp(a_{2, d_{1}})\}_{d_{1}}$, expressed by 
\begin{equation}\label{230122e_18}
    \sum_{2\le d_{1}\le d_0} a_{2, d_{1}}(\circu_{\mathfrak{m}_0})
    \mathfrak{a}^{(0)}(\theta; \bfu_{\mathfrak{m}_0}; \bxi)\equiv \mathfrak{a}^{(0)}(\theta; \bfu_{\mathfrak{m}_0}; \bxi)
\end{equation}
such that on each piece either it holds that
\begin{equation}\label{230114e3_15}
    \anorm{
    \frac{\partial^{d_{1}}}{\partial \theta^{d_{1}}}\Phi(\theta; \circu_{\mathfrak{m}_0}; \bxi)
    }\ge c_{d_{1}, 1}>0, \,\,\quad \forall \theta\in \supp(\mathfrak{a}^{(0)}),
\end{equation}
and 
\begin{equation}\label{221123e2_14}
    \frac{\partial^{d_{1}-1}}{\partial \theta^{d_{1}-1}}\Phi(\theta; \circu_{\mathfrak{m}_0}; \bxi)=0
\end{equation}
admits a unique solution on the interval $2\supp_1(\mathfrak{a}^{(0)})$, for some $2<d_{1}\le d_0$ and some $c_{d_{1}, 1}>0$, or \eqref{230114e3_15} holds for $d_{1}=2$.

Let us compare the contributions from these terms. As we have a sum of $O_d(1)$ terms, we only need to handle the terms that gives the biggest contribution. Assume that it is the term that corresponds to $d_1$ that gives the biggest contribution. \\

If $d_{1}=2$ that contributes most, then we just define 
\begin{equation}\label{230210e2_15z}
    \mathfrak{k}_1:=k, \ \ 
    \mathfrak{m}_1:=d_0-1, \ \ \mathfrak{d}_1:=d_1, \ \ \mathfrak{s}_1:=1, \ \ \Phi_{\mathfrak{s}_1}(\theta; \circu_{\mathfrak{m}_1}; \bxi):=\Phi(\theta; \circu_{d_0-1}; \bxi). 
\end{equation}
Here $\mathfrak{m}_1$ refers to the missing degree of the phase function $\Phi_{\mathfrak{s}_1}$, $\mathfrak{d}_1$ will be called the derivative degree of $\Phi_{\mathfrak{s}_1}$ as the $\mathfrak{d}_1$-th derivative of it is bounded from below, and $\mathfrak{s}_1$ is a scale parameter whose meaning will become clear later. Moreover, define 
\begin{equation}\label{230210e2_13}
    \mathfrak{a}^{(1)}(\theta; \circu_{\mathfrak{m}_1}; \bxi):=
    a_{2, d_1}(\circu_{\mathfrak{m}_0})
    \mathfrak{a}^{(0)}(\theta; \circu_{\mathfrak{m}_0}; \bxi).
\end{equation}
We will not  process the phase function $\Phi_{\mathfrak{s}_1}(\theta; \circu_{\mathfrak{m}_1}; \bxi)$, and will directly prove \eqref{230209e2_8} with $\mathfrak{a}^{(0)}$ replaced by $\mathfrak{a}^{(1)}$. Our algorithm terminates. \\

Assume $d_{1}>2$. Let $\theta_{d_{1}-1}=\theta_{d_{1}-1}(\circu_{
\mathfrak{m}_0
})$ be the unique solution in \eqref{221123e2_14}.  Let $c_{d_{1}, 2}$ be a small constant satisfying 
\begin{equation}\label{230122e3_21}
    c_{d_{1}, 2}\ll c_{d_{1}, 1}. 
\end{equation}
We cut the interval $2\supp_1(\mathfrak{a}^{(0)})$ into $O_d(1)$ many small intervals $I_{d_{1}}$ of length $c_{d_{1}, 2}$. \footnote{This simple step turns out to be quite fundamental; it will be used to get rid of ``global" zeros of the first order derivative of the phase function, and capture only the ``local" zeros. } We write this step of cutting as 
\begin{equation}
    \sum_{I_{d_{1}}} a_{1, I_{d_{1}}}(\theta)=1, \ \ \forall \theta,
\end{equation}
and each $a_{1, I_{d_{1}}}$ is supported on $2I_{d_{1}}$. Moreover, we will cut $\supp(a_{2, d_{1}})$ into $O_d(1)$ pieces, expressed by 
\begin{equation}
    a_{2, d_{1}}(\circu_{\mathfrak{m}_0})=\sum_{ I_{d_{1}}} a_{2, d_{1}, I_{d_{1}}}(\circu_{\mathfrak{m}_0}),
\end{equation}
so that for each $\circu_{\mathfrak{m}_0}\in \supp(a_{2, d_{1},  I_{d_{1}}})$, it holds that the equation \eqref{221123e2_14} admits a unique solution on $2 I_{d_{1}}$. So far we have 
\begin{equation}
    \pnorm{
    \sum_{|I_{d_{1}}|=c_{d_{1}, 2}} a_{1, I_{d_{1}}}(\theta)
    }
    \pnorm{
    \sum_{|\widetilde{I}_{d_{1}}|=c_{d_{1}, 2}} 
    a_{2, d_{1}, \widetilde{I}_{d_{1}}}(\circu_{\mathfrak{m}_0})
    }=1, \ \ \forall \theta, \ \circu_{\mathfrak{m}_0}\in \supp(a_{2, d_{1}}). 
\end{equation}
Note that on the support of $a_{1, I_{d_{1}}}(\theta) a_{2, d_{1}, \widetilde{I}_{d_{1}}
}(\circu_{\mathfrak{m}_0})$ with $\dist(I_{d_{1}}, 
\widetilde{I}_{d_{1}}
)\ge c_{d_{1}, 2}$, it always holds that
\begin{equation}
    \anorm{
    \frac{
    \partial^{d_{1}-1}
    }{
    \partial \theta^{d_{1}-1}
    }
    \Phi(\theta; \circu_{\mathfrak{m}_0}; \bxi)
    }\gtrsim 1. 
\end{equation}
For these terms, we can proceed exactly in the same way as in \eqref{230122e_18}--\eqref{221123e2_14}; we leave out the details. \\

It remains to handle the case where $I_{d_{1}}$ and $\widetilde{I}_{d_{1}}$ are either equal or adjacent. We without loss of generality assume that they are equal. In other words, from now on, we are only concerned with $(\theta; \circu_{\mathfrak{m}_0})$ that lies in the support of 
\begin{equation}
    a_{1, I_{d_{1}}}(\theta) a_{2, d_{1}, I_{d_{1}}
}(\circu_{\mathfrak{m}_0}).
\end{equation}
Denote 
\begin{equation}\label{230115e3_17}
    \Psi_{d_{1}-1, \iota}(\circu_{\mathfrak{m}_0}):=
    \Big(
    \frac{\partial^{\iota}}{\partial \theta^{\iota}}
    \Phi
    \Big)
    (\theta_{d_{1}-1}(\circu_{\mathfrak{m}_0}); \circu_{\mathfrak{m}_0}; (\xi, 1)),
\end{equation}
for $2\le \iota\le d+1$, and 
\begin{equation}
    \Psi_{d_{1}-1, 1}(\circu_{\mathfrak{m}_0}):=\frac{\partial \Phi}{\partial \theta}(\theta_{d_{1}-1}(\circu_{\mathfrak{m}_0}); \circu_{\mathfrak{m}_0}; (\xi, 1))-\xi. 
\end{equation}
Let $a: \R\to \R$ be a smooth compactly supported function, supported away from the origin; let $a_0: \R\to \R$ be a similar function but supported around the origin. Let
\begin{equation}\label{221124e2_17}
    s_{d_{1}, 0}:=2^{-\mathfrak{k}_0/d_{1}}.
\end{equation}
We without loss of generality assume that $s_{d_{1}, 0}$ is a dyadic number. For dyadic numbers $s_{d_{1}}> s_{d_{1}, 0}$, define
\begin{equation}\label{230123e4_34_Yang}
    \widetilde{\mathfrak{a}}_{d_{1}, s_{d_{1}}}(\circu_{\mathfrak{m}_0}; \bxi):=a_{2, d_{1}, I_{d_{1}}
}(\circu_{\mathfrak{m}_0})
    a\pnorm{
    \anorm{
    \frac{\Psi_{d_{1}-1, d_{1}-2}(\circu_{\mathfrak{m}_0})}{(s_{d_{1}})^2}
    }^2+\dots
    +
    \anorm{
    \frac{\Psi_{d_{1}-1, 1}(\circu_{\mathfrak{m}_0})+\xi'}{(s_{d_{1}})^{d_{1}-1}}
    }^2
    }, \ \ \xi':=\xi/\eta,
\end{equation}
and
\begin{equation}\label{221124e2_18}
    \mathfrak{a}_{d_{1}, s_{d_{1}}}(\theta; \circu_{\mathfrak{m}_0}; \bxi):=
    a_{1, I_{d_{1}}}(\theta)
    \widetilde{\mathfrak{a}}_{d_{1}, s_{d_{1}}}(\circu_{\mathfrak{m}_0}; \bxi)\mathfrak{a}^{(0)}(\theta; \circu_{\mathfrak{m}_0}; \bxi).
\end{equation}
For $s_{d_{1}}=s_{d_{1}, 0}$, define $\mathfrak{a}_{d_{1}, s_{d_{1}}}$ similarly to \eqref{221124e2_18} but with $a_0$ in place of $a$. After the decomposition in \eqref{230123e4_34_Yang}, the multiplier we are concerned with can be written as \begin{equation}\label{230209e2_25}
    \sum_{s_{d_1}\ge s_{d_1, 0}}
    \int_{\R} e^{i2^{\mathfrak{k}_0} \Phi(\theta; \circu_{\mathfrak{m}_0}; \bxi)} \mathfrak{a}_{d_1, s_{d_1}}(\theta; \circu_{\mathfrak{m}_0}; \bxi)d\theta.
\end{equation}
The choice of the scale $s_{d_{1}, 0}$ will become clear in \eqref{221124e2_38}. As we have $O(k)$ many terms in the sum over $s_{d_1}$, and we are allowed to lose $2^{\epsilon k}$ in the desired estimate, we therefore only need to consider the $s_{d_1}$ term that gives the biggest contribution. \\

Let $c_{d_{1}, 3}$ be a small constant satisfying 
\begin{equation}\label{230122e3_32}
    c_{d_{1}, 3}\ll c_{d_{1}, 2}. 
\end{equation}
It turns out that the case $s_{d_{1}}\ge c_{d_{1}, 3}$ and $s_{d_{1}}\le c_{d_{1}, 3}$ should be handled differently. Let us start with the former case. In this case, we should not think we are in the case where 
\begin{equation}
    \Big|\frac{\partial^{d_{1}}}{\partial \theta^{d_{1}}} \Phi(\theta; \circu_{\mathfrak{m}_0}; \bxi)\Big|\gtrsim 1,
\end{equation}
but rather 
\begin{equation}\label{230118e3_26}
    \Big|\frac{\partial^{d_{1}'}}{\partial \theta^{d_{1}'}} \Phi(\theta; \circu_{\mathfrak{m}_0}; \bxi)\Big|\gtrsim 1,
\end{equation}
for some $d_{1}'< d_{1}$. Let us be more precise. The assumption that $s_{d_{1}}\simeq 1$ says that when the $(d_{1}-1)$-th derivative of the phase function in $\theta$ vanishes, there exists $d_{1}'<d_{1}-1$ such that 
\begin{equation}
    \Big|\frac{\partial^{d_{1}'}}{\partial \theta^{d_{1}'}} \Phi(\theta; \circu_{\mathfrak{m}_0}; \bxi)\Big|\gtrsim 1,\quad \forall |\theta-\theta_{d_{1}-1}|\ll 1.
\end{equation}
This allows us to cut the range of $\theta$ into small intervals (of length still comparable to 1) such that on each of these small intervals we have \eqref{230118e3_26} for some $d_{1}'<d_{1}$. 
In each of these cases, we argue in exactly the same way as in \eqref{230114e3_15}, \eqref{221123e2_14} and the paragraph below them. The details are left out. \\

From now on, we assume that $s_{d_{1}}\le c_{d_{1}, 3}$. Now we are ready to define 
\begin{equation}\label{230210e2_32}
     \mathfrak{d}_1:=d_1. 
\end{equation}
We define $\mathfrak{d}_1$ here but not earlier as there are cases where our phase function may satisfy better derivative bounds as in \eqref{230118e3_26}.

\begin{claim}\label{221031claim3_1}
Let $(\theta; \circu_{\mathfrak{m}_0}; \bxi)\in \supp(\mathfrak{a}_{d_{1}, s_{d_{1}}})$. Then 
\begin{equation}
    \anorm{
    \frac{
    \partial
    }{\partial \theta}
    \Phi(\theta; \circu_{\mathfrak{m}_0}; \bxi)
    }\gtrsim |\theta-\theta_{d_{1}-1}|^{d_{1}-1},
\end{equation}
whenever 
\begin{equation}
    |\theta-\theta_{d_{1}-1}|\ge (100d!)c_{d_1,1}^{-1} s_{d_{1}}.
\end{equation}
\end{claim}
\begin{proof}[Proof of Claim \ref{221031claim3_1}]
If we do a Taylor expansion for $\frac{\partial \Phi}{\partial \theta}$ about the point $\theta_{d_{1}-1}$, then it  looks like 
\begin{equation}\label{230122e3_38}
    w_{d_{1}-1}
    (\Delta\theta)^{d_{1}-1}+
    \dots+w_1 (\Delta\theta)+w_0+O(|\Delta\theta|^{d_1}),
\end{equation}
where $\Delta\theta=\theta-\theta_{d_{1}-1}$, and 
\begin{equation}
     |w_{d_{1}-1}|\ge \frac{c_{d_{1}, 1}}{(d_1-1)!}, \ \  w_{d_{1}-2}=0, \ \ 
    |w_{d_{1}-3}|\leq \frac{(s_{d_{1}})^2}{(d_1-3)!}, |w_{d_{1}-4}|\leq \frac{(s_{d_{1}})^3}{(d_1-4)!}, \dots, |w_0|\leq (s_{d_{1}})^{d_{1}-1}. 
\end{equation}
Recall the choices of constants in \eqref{230122e3_21} and \eqref{230122e3_32}. 
It is elementary to see that we have the desired properties if the $\ll$ there were chosen to be small enough compared to the implicit constant in the last term of \eqref{230122e3_38}. 
\end{proof}

Claim \ref{221031claim3_1} suggests that we further truncate the  integral in the $\theta$ variable. Let us be more precise. Let $\epsilon_{d_{1}}:=\epsilon (10d_{1})^{-1}$. 
Let $\varphi^+_{s_{d_{1}}}: \R\to \R$ be an $L^{\infty}$-normalized smooth bump function adapted to the interval
\begin{equation}\label{221124e2_35}
    [-2^{\epsilon_{d_{1}}\cdot k}\cdot s_{d_{1}}, 2^{\epsilon_{d_{1}}\cdot k}\cdot s_{d_{1}}]. 
\end{equation}
Write the multiplier in \eqref{230209e2_25}  as 
\begin{equation}\label{221124e2_36}
\begin{split}
    & \int_{\R} e^{i2^{\mathfrak{k}_0}
    \Phi(\theta; \circu_{\mathfrak{m}_0}; \bxi)}
    \mathfrak{a}_{d_{1}, s_{d_{1}}}(\theta; \circu_{\mathfrak{m}_0}; \bxi)
    \varphi^+_{s_{d_{1}}}(\theta-\theta_{d_{1}-1}(\circu_{\mathfrak{m}_0}))d\theta\\
    & + \int_{\R} e^{i2^{\mathfrak{k}_0}
    \Phi(\theta; \circu_{\mathfrak{m}_0}; \bxi)}
    \mathfrak{a}_{d_{1}, s_{d_{1}}}(\theta; \circu_{\mathfrak{m}_0}; \bxi)
    (
    1-\varphi^+_{s_{d_{1}}}(\theta-\theta_{d_{1}-1}(\circu_{\mathfrak{m}_0}))
    )d\theta.
\end{split}
\end{equation}
Because of the $\epsilon$-room $2^{\epsilon_{d_{1}}\cdot k}$ we created in \eqref{221124e2_35}, the second term on the right hand side of \eqref{221124e2_36} decays rapidly. To see this, note that there we have  
\begin{equation}\label{221124e2_38}
    \anorm{
    \frac{\partial (2^{k} \Phi)}{\partial \theta}
    }\gtrsim 2^{k} (2^{\epsilon_{d_{1}}\cdot k}\cdot s_{d_{1}})^{d_{1}-1}.
\end{equation}
Moreover, each time when we apply integration by parts, we collect a factor $(2^{\epsilon_{d_{1}}\cdot k}\cdot s_{d_{1}})^{-1}$, which is the derivative of the involved cut-off function. This explains the choice of $s_{d_{1}, 0}$ as in \eqref{221124e2_17}.\\

Let us focus on the first term on the right hand side of \eqref{221124e2_36}. Let $\varphi_{s_{d_{1}}}: \R\to \R$ be an $L^{\infty}$-normalized smooth bump function adapted to the interval $[-s_{d_{1}}, s_{d_{1}}]$.  By losing a multiplicative constant $2^{\epsilon_{d_{1}}\cdot k}$, we will only consider the contribution from
\begin{equation}\label{221124e2_35zz}
\iint_{\R^2} \widehat{f}(\bxi)
m_{d_{1}, s_{d_{1}}}(\circu_{\mathfrak{m}_0}; \bxi) e^{i\bfx\cdot \bxi} d\bxi, 
\end{equation}
where 
\begin{equation}\label{230123e_4_30_Yang}
    m_{d_1, s_{d_1}}(\circu_{\mathfrak{m}_0}; \bxi):=\int_{\R} e^{i2^{\mathfrak{k}_0} \Phi(\theta; \circu_{\mathfrak{m}_0}; \bxi)}
    \mathfrak{a}_{d_{1}, s_{d_{1}}}(\theta; \circu_{\mathfrak{m}_0};\bxi) \varphi_{s_{d_{1}}}(\theta-\theta_{d_{1}-1}(\circu_{\mathfrak{m}_0}))d\theta.
\end{equation}

Note that the partial derivative of 
\begin{equation}\label{Y_230407}
\anorm{
    \frac{\Psi_{d_{1}-1, d_{1}-2}}{(s_{d_{1}})^2}
    }^2+\dots
    +
    \anorm{
    \frac{\Psi_{d_{1}-1, 1}+\xi'}{(s_{d_{1}})^{d_{1}-1}}
    }^2
\end{equation}
in the $\xi$ variable is comparable to $(s_{d_{1}})^{1-d_{1}}$. This suggests that we further decompose our multiplier as 
\begin{equation}
m_{d_{1}, s_{d_{1}}}(\circu_{\mathfrak{m}_0}; \bxi)
=\sum_{\Theta_{d_{1}}\subset \R, \ell(\Theta_{d_{1}})=(s_{d_{1}})^{d_{1}-1}
} m_{d_{1}, s_{d_{1}}, \Theta_{d_{1}}}(\circu_{\mathfrak{m}_0}; \bxi),
\end{equation}
where $m_{d_{1}, s_{d_{1}}, \Theta_{d_{1}}}(\circu_{\mathfrak{m}_0}; \bxi)$ is defined by the amplitude function $\mathfrak{a}_{d_{1}, s_{d_{1}}}(\theta; \circu_{\mathfrak{m}_0}; \bxi) a_{3, \Theta_{d_{1}}}(\xi')$, 
for appropriately defined smooth bump function $a_{3, \Theta_{d_{1}}}$ adapted to the interval $\Theta_{d_{1}}$. We further bound the $L^p$ norm of  \eqref{221124e2_35zz} by 
\begin{equation}\label{221124e2_33} 
\Big( 
\sum_{\Theta_{d_{1}}}
\Norm{
\iint_{\R^2} \widehat{f}(\bxi)
m_{d_{1}, s_{d_{1}}, \Theta_{d_{1}}}(\circu_{\mathfrak{m}_0}; \bxi) e^{i\bfx\cdot \bxi} d\bxi 
}_{L^p}^p \Big)^{\frac{1}{p}}.
\end{equation}
This holds since the disjointness of $\Theta_{d_1}$ in the $\xi'$ variable implies the disjointness of $\bfu_{\mathfrak m_0}$, in view of the last term of \eqref{Y_230407}.

To analyze the multiplier $m_{d_{1}, s_{d_{1}}, \Theta_{d_{1}}}(\circu_{\mathfrak{m}_0}; \bxi)$, we first do a change of variables
\begin{equation}
    \theta\mapsto \theta+\theta_{d_{1}-1}(\circu_{\mathfrak{m}_0}).
\end{equation}
The new phase function is 
\begin{align}
&\Phi(\theta+\theta_{d_{1}-1}(\circu_{\mathfrak{m}_0}); \circu_{\mathfrak{m}_0}; \bxi)\\
&=\Phi(\theta_{d_1-1}(\circu_{\mathfrak{m}_0});\circu_{\mathfrak{m}_0}; \bxi)
+\Big(\eta\Psi_{d_{1}-1, 1}(\circu_{\mathfrak{m}_0})+\xi\Big)\theta
+\eta\sum_{\iota=2}^d \Psi_{d_{1}-1, \iota}(\circu_{\mathfrak{m}_0}) \frac{\theta^{\iota}}{\iota!} + \eta \theta^{d+1}P(\theta;\circu _{\mathfrak m_0};\bxi),
\end{align}
where $P$ is a smooth function.

Note that $\Psi_{d_{1}-1, \iota}$ is constant in $\bxi$ whenever $\iota\ge 1$.  We continue with the change of variables 
\begin{equation}\label{221119e4_24}
    \Psi_{d_{1}-1, d}(\circu_{\mathfrak{m}_0})\mapsto v_d, \ \dots, \  \Psi_{d_{1}-1, d_{1}}(\circu_{\mathfrak{m}_0}) \mapsto v_{d_{1}}, \ \Psi_{d_{1}-1, d_{1}-2}(\circu_{\mathfrak{m}_0})\mapsto v_{d_{1}-2}, \ \dots, \  \Psi_{d_{1}-1, 1}(\circu_{\mathfrak{m}_0})\mapsto v_1.
\end{equation}
This change of variables can also be written as 
\begin{equation}\label{230603e4_49}
    (\Psi_{d_{1}-1, 1}(\circu_{\mathfrak{m}_0}), \dots, \Psi_{d_{1}-1, d}(\circu_{\mathfrak{m}_0}))\mapsto \bfv_{d_{1}-1},
\end{equation}
where $\bfv_{d_{1}-1}=(v_1, \dots, v_{d_{1}-2}, v_{d_{1}}, \dots, v_d)$. Note that in the above change of variables, we do not have the term $\Psi_{d_{1}-1, d_{1}-1}$, as it vanishes constantly. 

\begin{claim}\label{221124claim2_4}
The Jacobian of the change of variables in \eqref{221119e4_24}  has absolute value comparable to 1. 
\end{claim}
\begin{proof}[Proof of Claim \ref{221124claim2_4}]
We only give a proof in the case $d_1<d+1$, and the case $d_1=d+1$ is similar with only notational changes. For all $i\ne \mathfrak m_0$, $j\ne d_1-1$, we have
\begin{equation}
    \partial_{u_i}v_j=
    \partial_{u_i}\Big( 
    \Psi_{d_1-1,j}
    \Big)
    =\Psi_{d_1-1,j+1}\partial_{u_i}\theta_{d_1-1}+\frac{\theta_{d_1-1}^{i-j}}{(i-j)!}+\Big(\partial_{u_i} \partial_{\theta}^j R
    \Big)(\theta_{d_1-1}; \circu_{\mathfrak{m}_0}),    
\end{equation}
where $R$ collects remainder terms
\begin{equation}
    R(\theta; \bfu_{\mathfrak{m}_0}):=Q(\bfu_{\mathfrak{m}_0})
    \frac{\theta^{\mathfrak{m}_0}}{\mathfrak{m}_0!}
    +
    P(\theta; \bfu_{
    \mathfrak{m}_0
    })\frac{\theta^{d+1}}{(d+1)!}.
\end{equation}
Recall that we have the convention $i!=\infty$ for negative integers $i$. Also recall that $\theta_{d_1-1}$ is the unique solution of 
\begin{equation}
    \frac{\partial^{d_1-1}}{\partial\theta^{d_1-1}}\Phi(\theta;\bfu_{\mathfrak m_0};\bxi)=0.
\end{equation}
To simplify the notation, we will abbreviate $\theta_{d_1-1}$ to $\theta$ in the rest of the  proof. Taking a partial derivative in $u_i$, we get
\begin{equation}
    \partial_{u_i}\theta=-\frac{\frac{\theta^{i-d_1+1}}{(i-d_1+1)!}+(\partial_{u_i} \partial_{\theta}^{d_1-1}
    R)(\theta; \circu_{\mathfrak{m}_0})}{\Psi_{d_1-1,d_1}}, \quad \forall \ i\neq \mathfrak m_0.
\end{equation}
The Jacobian we need to compute can be writen as 
\begin{equation}
    \mathrm{Jacobian}
    =\det \left(\Psi_{d_1-1,j+1}\partial_{u_i}\theta+\frac{\theta^{i-j}}{(i-j)!}+
    \partial_{
    u_i
    }
    \partial^j_{\theta}
    R\right)_{i\ne \mathfrak m_0,j\ne d_1-1}(\theta; \bfu_{\mathfrak m_0}).
\end{equation}
Splitting each column of the determinant into two terms, and by elementary column operations, we have
\begin{equation}
    \mathrm{Jacobian}    =\det\left(\frac{\theta^{i-j}}{(i-j)!}+\partial_{u_i}
    \partial^j_{\theta}
    R\right)_{i\ne \mathfrak m_0,j\ne d_1-1}
    +
    \sum_{k\ne d_1-1}\det\left(a^{(k)}_{ij}\right)_{i\ne \mathfrak m_0,j\ne d_1-1},
\end{equation}
where 
\begin{equation}
    a^{(k)}_{ij}=
    \begin{cases}
        \displaystyle\partial_{u_i}
        \partial^j_{\theta}
        R+\frac{\theta^{i-j}}{(i-j)!}\quad &\text{if }j\ne k,\\
        \displaystyle-\frac{\Psi_{d_1-1,j+1}}{\Psi_{d_1-1,d_1}}\left(\frac{\theta^{i-d_1+1}}{(i-d_1+1)!}+\partial_{u_i}
        \partial^{d_1-1}_{\theta}
        R\right)\quad &\text{if }j= k.
    \end{cases}       
\end{equation}
By co-factor expansion, the Jacobian is thus equal to the following $d\times d$ determinant:
\begin{equation}
    \det\begin{bmatrix}
        (-1)^{d_1}\frac{\Psi_{d_1-1,2}}{\Psi_{d_1-1,d_1}} & (-1)^{d_1}\frac{\Psi_{d_1-1,3}}{\Psi_{d_1-1,d_1}} & \dots & (-1)^{d_1}\frac{\Psi_{d_1-1,d+1}}{\Psi_{d_1-1,d_1}}\\
        1+\partial_{u_1}
        \partial_{\theta}
        R &\partial_{u_1}
        \partial^2_{\theta}
        R & \dots &\partial_{u_1}
        \partial^d_{\theta}
        R\\        
        \theta+\partial_{u_2}
        \partial_{\theta}
        R & 1+\partial_{u_2}
        \partial^2_{\theta}
        R & \dots & \partial_{u_2}
        \partial^d_{\theta}
        R\\
        \vdots & \vdots & \ddots &\vdots\\
        \frac{\theta^{\mathfrak m_0-2}}{(\mathfrak m_0-2)!}+\partial_{u_{\mathfrak m_0-1}}
        \partial_{\theta}
        R & \frac{\theta^{\mathfrak m_0-3}}{(\mathfrak m_0-3)!}+\partial_{u_{\mathfrak m_0-1}}
        \partial^2_{\theta}
        R & \dots & \frac{\theta^{\mathfrak m_0-d-1}}{(\mathfrak m_0-d-1)!}+\partial_{u_{\mathfrak m_0-1} }
        \partial^d_{\theta}
        R\\

        \frac{\theta^{\mathfrak m_0}}{(\mathfrak m_0)!}+\partial_{u_{\mathfrak m_0+1}}
        \partial_{\theta}
        R & \frac{\theta^{\mathfrak{m}_0-1}}{(\mathfrak{m}_0-1)!}+\partial_{u_{\mathfrak m_0+1}}
        \partial^2_{\theta}
        R &\dots & \frac{\theta^{\mathfrak m_0-d+1}}{(\mathfrak m_0-d+1)!}+\partial_{u_{\mathfrak m_0+1} }
        \partial_{\theta}^{d}
        R\\
        \vdots & \vdots & \ddots & \vdots\\
        \frac{\theta^{d-1}}{(d-1)!}+\partial_{u_d}
        \partial_{\theta}
        R & \frac{\theta^{d-2}}{(d-2)!}+\partial_{u_d}
        \partial^2_{\theta}
        R & \dots & 1+\partial_{u_d}
        \partial^d_{\theta}
        R
    \end{bmatrix}.
\end{equation}
On the other hand, we note this $d\times d$ determinant is equal to 
\begin{equation}
    \frac{(-1)^{d_1}}{\Psi_{d_1-1,d_1}}\det\begin{bmatrix}
        \Psi_{d_1-1,2} & \Psi_{d_1-1,3} & \dots & \Psi_{d_1-1,d+1}\\     \partial_{u_1}
        \partial_{\theta}
        \gamma & \partial_{u_1}
        \partial^2_{\theta}
        \gamma& \dots &\partial_{u_1}
        \partial^d_{\theta}
        \gamma\\
        \vdots & \vdots & \ddots & \vdots\\

        \partial_{u_{\mathfrak m_0-1}}
        \partial_{\theta}
        \gamma & \partial_{u_{\mathfrak m_0-1}}
        \partial^2_{\theta}
        \gamma & \dots &\partial_{u_{\mathfrak m_0-1}}
        \partial^d_{\theta}
        \gamma\\
        
        \partial_{u_{\mathfrak m_0+1}}
        \partial_{\theta}
        \gamma & \partial_{u_{\mathfrak m_0+1}}
        \partial^2_{\theta}
        \gamma & \dots &\partial_{u_{\mathfrak m_0+1}}
        \partial^d_{\theta}
        \gamma\\
        \vdots & \vdots & \ddots & \vdots\\
        \partial_{u_d}
        \partial_{\theta}
        \gamma & \partial_{u_d}
        \partial^2_{\theta}
        \gamma & \dots & \partial_{u_d}
        \partial^d_{\theta}
        \gamma
        \end{bmatrix},
\end{equation}
where the determinant right above is exactly the non-degeneracy condition \eqref{Y_230330nondegenerate} at $(\theta; \circu_{\mathfrak{m}_0})$, except with a transpose and $\bfu_{\mathfrak m_0}$ in place of $\bfv$. By continuity, this is nonzero if the initial bump function $\chi(\theta;\bfv)$ is supported in a small enough neighbourhood of the origin. Also, $|\Psi_{d_1-1,d_1}|\simeq 1$ which follows from \eqref{230114e3_15}. Thus we are done.
\end{proof}

After the change of variables in \eqref{221119e4_24}, what we need to control becomes
\begin{equation}\label{230114e3_48}
       \Big(
    \sum_{\Theta_{d_{1}}}
    \Norm{
    \iint_{\R^2}
    \widehat{f}(\bxi) 
    \bnorm{
    \int_{\R} e^{i2^{
    \mathfrak{k}_0
    } \Phi(\theta; \circv_{d_{1}-1}; \bxi)}
    \varphi_{s_{d_{1}}}(\theta)\mathfrak{a}^{(0)}(\theta; \circu_{\mathfrak{m}_0}; \bxi)d\theta
    }
    \mathfrak{b}_{d_{1}, s_{d_{1}}, \Theta_{d_{1}}}(\circv_{d_{1}-1}; \bxi)
    e^{i(\bfx\cdot \bxi)}d\bxi
    }_{L^p_{\circv_{d_{1}-1}; \bfx}
    }^p\Big)^{\frac{1}{p}} 
\end{equation}
where (using \eqref{230123e_4_30_Yang} and \eqref{230123e4_34_Yang}) 
\begin{equation}
    \mathfrak{b}_{d_{1}, s_{d_{1}}, \Theta_{d_{1}}}(\circv_{d_{1}-1}; \bxi):=\widetilde{\mathfrak{a}}_{d_{1}, s_{d_{1}}}(\circu_{\mathfrak{m}_0}(\circv_{d_{1}-1}); \bxi) a_{3, \Theta_{d_{1}}}(\xi'),
\end{equation}
with $\circu_{\mathfrak{m}_0}(\circv_{d_{1}-1})$ being the reverse of \eqref{221119e4_24}, and 
\begin{equation}\label{230114e3_50}
    \Phi(\theta; \circv_{d_{1}-1}; \bxi):=\eta\Big(v_d \frac{\theta^d}{d!}+\dots+v_{d_{1}}
    \frac{\theta^{d_{1}}}{(d_1)!}
    +v_{d_{1}-2}\frac{\theta^{d_{1}-2}}{(d_1-2)!}+\dots+v_2 \frac{\theta^2}{2!}\Big)+(v_1\eta+\xi)\theta+O(|\theta|^{d+1}).
\end{equation}
Let us record here that (using \eqref{230123e4_34_Yang} and \eqref{221119e4_24}, and with $a$ replaced by $a_0$ if $s_{d_{1}}=s_{d_{1},0}$)
\begin{equation}\label{230208e2_47}
    \mathfrak{b}_{d_{1}, s_{d_{1}}, \Theta_{d_{1}}}(\circv_{d_{1}-1}; \bxi)=a_{2, d_{1}, I_{d_{1}}}(\circu_{\mathfrak{m}_0}(\circv_{d_{1}-1})) 
    a\pnorm{
    \anorm{
    \frac{v_{d_{1}-2}}{(s_{d_{1}})^2}
    }^2+\dots+
    \anorm{
    \frac{v_{1}+\xi'}{(s_{d_{1}})^{d_{1}-1}}
    }^2
    } a_{3, \Theta_{d_{1}}}(\xi').
\end{equation}
\begin{remark}
    The only role that the amplitude function $a_{2, d_{1}, I_{d_{1}}}(\circu_{\mathfrak{m}_0}(\circv_{d_{1}-1}))$ plays is that it tells us $|v_{d_{1}}|\gtrsim 1$ (when we are in the case $d_1<d+1$). 
\end{remark}
Let $c(\Theta_{d_{1}})$ be the center of the interval $\Theta_{d_{1}}$.  We apply the change of variables $\theta\mapsto s_{d_{1}}\theta$, then 
\begin{equation}\label{Y_230412_114am}    
v_{d_{1}-2}\mapsto v_{d_{1}-2} (s_{d_{1}})^2, \ \ \dots, \ \  v_2\mapsto v_2 (s_{d_{1}})^{d_{1}-2}, \ \ v_1+c(\Theta_{d_{1}})\mapsto v_1 (s_{d_{1}})^{d_{1}-1},
\end{equation}
and in the end 
\begin{equation}\label{230411e3_63}
    \eta\mapsto \eta, \ \ 
    \frac{\xi-c(\Theta_{d_{1}}) \eta}{(s_{d_{1}})^{d_{1}-1}}
    \mapsto \xi.
\end{equation}
These will turn \eqref{230114e3_48} to  
\begin{equation}\label{230114e3_53}
    s_{d_{1}} (s_{d_{1}})^{
    \frac{(d_{1}+1)(d_{1}-2)}{2p}
    } 
    \jac(s_{d_{1}})
    \Big( 
    \sum_{\Theta_{d_{1}}}
    \Norm{
    \iint_{\R^2}
    \widehat{f}_{\Theta_{d_{1}}}(\bxi) 
    \bnorm{
    \int_{\R} e^{i 2^{k_1}
    \Phi_{s_{d_{1}}}(\theta; \circv_{d_{1}-1}; \bxi)}
    \varphi(\theta)d\theta
    }
    \widetilde{\mathfrak{b}}_{d_{1}, s_{d_{1}}}(\circv_{d_{1}-1}; \bxi)
    e^{i(\bfx\cdot \bxi)}d\bxi
    }_{L^p
    }^p\Big)^{\frac{1}{p}}, 
\end{equation}
where 
\begin{equation}\label{Y_230520}
\widehat f_{\Theta_{d_1}}(\bxi):=\widehat f\left((s_{d_{1}})^{d_{1}-1}\xi+c(\Theta_{d_{1}}) \eta,\eta\right),   
\end{equation}
\begin{equation}\label{230208e2_51}
    2^{k_1}:=2^{\mathfrak{k}_0}  (s_{d_{1}})^{d_{1}}, \ \ \jac(s_{d_{1}}):=(s_{d_{1}})^{d_{1}-1} (s_{d_{1}})^{-\frac{d_{1}-1}{p}},
\end{equation}
\begin{equation}\label{230123e3_64}
    \widetilde{\mathfrak{b}}_{d_{1}, s_{d_{1}}}(\circv_{d_{1}-1}; \bxi)=\widetilde{a}_{2, d_{1}, I_{d_{1}}}(\circv_{d_{1}-1}) 
    a\pnorm{
    (v_{d_{1}-1})^2+\dots+
    (v_1+\xi')^2
    },
\end{equation}
with 
\begin{equation}
    \widetilde{a}_{2, d_{1}, I_{d_{1}}}(\circv_{d_{1}-1}):=a_{2, d_{1}, I_{d_{1}}}(\circu_{\mathfrak{m}_0}(\circv_{d_{1}-1})),
\end{equation}
and 
\begin{equation}\label{230114e3_55}
    \Phi_{s_{d_{1}}}
    (\theta; \circv_{d_{1}-1}; \bxi):=\eta
    \pnorm{
    v_{d} (s_{d_{1}})^{d-d_{1}}\frac{\theta^d}{d!}+\dots+v_{d_{1}} (s_{d_{1}})^{d_{1}-d_{1}}\frac{\theta^{d_{1}}}{(d_{1})!}+ 
    v_{d_{1}-2} \frac{\theta^{d_{1}-2}}{(d_{1}-2)!}+\dots+v_1 \theta
    }+\xi\theta+O(|\theta|^{d+1}).
\end{equation}

If we denote 
\begin{equation}
    \Vec{D}_{d_1-1, s_{d_1}}(\circv_{d_1-1}):=(v_d (s_{d_1})^{d-d_1}, \dots, 
    v_{d_1} (s_{d_1})^{d_1-d_1}, 0, v_{d_1-2}, \dots, v_1
    )\in \R^{d},
\end{equation}
then the phase function in \eqref{230114e3_55} can be written as 
\begin{equation}
    \Phi_{s_{d_1}}(\theta; \circv_{d_1-1}; \bxi)=\eta \Vec{D}_{d_1-1, s_{d_1}}(\circv_{d_1-1})\cdot (
    \frac{\theta^d}{d!}, \dots, \frac{\theta}{1!}
    )+\xi\theta+O(|\theta|^{d+1}). 
\end{equation}
The factor $\jac(s_{d_{1}})$ in \eqref{230208e2_51} appears when doing the change of variables in $\bxi$ and then in $\bfx$. It is not important, as it will be cancelled out later when we revert the above changes of variables. \\

We are ready to define the resulting data as in \eqref{230210e2_13} and  \eqref{230210e2_15z}. Recall from \eqref{230210e2_32} we have already defined $\mathfrak{d}_1=d_1$. Moreover, define
\begin{equation}\label{230210e2_63} 
    \mathfrak{k}_1:=k_1, \ \ 
    \mathfrak{m}_1:=d_1-1, \ \  \mathfrak{s}_1:=s_{d_1}, \ \ \Phi_{\mathfrak{s}_1}(\theta; \circu_{\mathfrak{m}_1}; \bxi):=\Phi_{s_{d_1}}(\theta; \circu_{d_1-1}; \bxi). 
\end{equation}
It remains to define amplitude functions $\mathfrak{a}^{(1)}(\theta; \circu_{\mathfrak{m}_1}; \bxi)$, like in \eqref{230210e2_13}. 

If we are in the case $s_{d_{1}}=s_{d_{1}, 0}$, then define 
\begin{equation}\label{230211e2_63}
    \mathfrak{a}^{(1)}(\theta; \circu_{\mathfrak{m}_1}; \bxi):= \varphi(\theta) \widetilde{\mathfrak{b}}_{d_1, s_{d_1}}(\circu_{d_1-1}; \bxi),
\end{equation}
and terminate the algorithm. 

Assume that $s_{d_1}> s_{d_1, 0}$.  Consider the new phase function $\Phi_{s_{d_{1}}}$ given by \eqref{230114e3_55}. Note that in the support of the amplitude function $\widetilde{a}_{2, d_{1}, I_{d_{1}}}$, it always holds that $|v_{d_{1}}|\simeq 1$ when $d_1<d+1$. If we are in the case $d_1=d+1$,  the coefficient of the term $\theta^{d+1}$ is not called $v_{d+1}$ anymore, but it still has an absolute value comparable to $1$.  Let $c_{d_{1}, 4}\ll c_{d_{1}, 3}$ be a small constant. We write the integral in $\theta$ in the expression \eqref{230114e3_53} as 
\begin{equation}\label{230208e2_56}
    \int_{\R} e^{i 2^{k_1}
    \Phi_{s_{d_{1}}}(\theta; \circv_{d_{1}-1}; \bxi)}
    \varphi_{c_{d_{1}, 4}}(\theta)d\theta+  \int_{\R} e^{i 2^{k_1}
    \Phi_{s_{d_{1}}}(\theta; \circv_{d_{1}-1}; \bxi)}
    (1-\varphi_{c_{d_{1}, 4}}(\theta))d\theta,
\end{equation}
where $\varphi_{c_{d_{1}, 4}}$ is supported on $[-c_{d_{1}, 4}, c_{d_{1}, 4}]$ and $(1-\varphi_{c_{d_{1}, 4}}(\theta))$ is supported away from the origin. Let us compare the contribution between the former term and the latter term in \eqref{230208e2_56}. If we are in the case that the latter term dominates, then we define 
\begin{equation}\label{230210e2_68}
    \mathfrak{a}^{(1)}(\theta; \circu_{\mathfrak{m}_1}; \bxi):= (1-\varphi_{c_{d_1, 4}}(\theta)) \widetilde{\mathfrak{b}}_{d_1, s_{d_1}}(\circu_{d_1-1}; \bxi).
\end{equation}
Otherwise, we define 
\begin{equation}\label{230211e2_66}
       \mathfrak{a}^{(1)}(\theta; \circu_{\mathfrak{m}_1}; \bxi):= \varphi_{c_{d_1, 4}}(\theta) \widetilde{\mathfrak{b}}_{d_1, s_{d_1}}(\circu_{d_1-1}; \bxi).
\end{equation}

The resulting term in \eqref{230114e3_53} can therefore be written as 
\begin{equation}\label{230210e2_61}
    \mathfrak{s}_1 (\mathfrak{s}_1)^{
    \frac{(\mathfrak{m}_1+2)(\mathfrak{m}_1-1)}{2p}
    } 
    \jac(\mathfrak{s}_1)
    \Big( 
    \sum_{\Theta_{d_{1}}}
    \Norm{
    \iint_{\R^2}
    \widehat{f}_{\Theta_{d_{1}}}(\bxi) 
    \bnorm{
    \int_{\R} e^{i 2^{\mathfrak{k}_1}
    \Phi_{\mathfrak{s}_1}(\theta; \circu_{\mathfrak{m}_1}; \bxi)}
    \mathfrak{a}^{(1)}(\theta; \circu_{\mathfrak{m}_1}; \bxi)d\theta
    }
    e^{i(\bfx\cdot \bxi)}d\bxi
    }_{L^p
    }^p\Big)^{\frac{1}{p}}.
\end{equation}
This finishes defining data for the first step of the algorithm. Recall the choice of data in \eqref{230210e2_15z}. Let us record that
\begin{equation}\label{230212e2_68}
    \mathfrak{s}_1\in 
    \begin{cases}
    \{1\}, & \text{ if  } \mathfrak{d}_1=2;\\
    [2^{-\frac{\mathfrak{k}_0}{\mathfrak{d}_1}} , \frac{1}{2}] & \text{ if } \mathfrak{d}_1>2, 
    \end{cases}
    \ \ \text{ and }  2^{\mathfrak{k}_1}=2^{\mathfrak{k}_0} (\mathfrak{s}_1)^{\mathfrak{d}_1}.
\end{equation}
We will repeat the whole argument in the first step for every term in \eqref{230210e2_61}, which will be called the second step of the algorithm. \\

Before we continue to the next step of the algorithm, let us explain the choices of the amplitude functions in \eqref{230211e2_63}, \eqref{230210e2_68} and \eqref{230211e2_66}. 

If we are in the $s_{d_1}=s_{d_1, 0}$, we see that $k_1=0$, and the integral in $\theta$ in \eqref{230210e2_61} does not oscillate anymore. Therefore, our algorithm will not further process the phase function.

Consider the amplitude function in \eqref{230210e2_68}.  We take the $(d_{1}-1)$-th derivative of the phase function $\Phi_{\mathfrak{s}_1}$ in the $\theta$ variable, and obtain 
\begin{equation}
    v_d (s_{d_{1}})^{d-d_1} \frac{\theta^{d-d_{1}+1}}{(d-d_{1}+1)!}+\dots+v_{d_{1}} \theta.
\end{equation}
Recall that $|s_{d_{1}}|\le c_{d_{1}, 3}$ and $\theta$ is away from the origin in this case. We therefore see that 
\begin{equation}
    \anorm{
    \frac{
    \partial^{d_{1}-1}}{\partial \theta^{d_{1}-1}}
    \Phi_{s_{d_{1}}}(\theta; \circv_{d_{1}-1}; \bxi)
    }\gtrsim 1. 
\end{equation}
Recall that we started with a phase function $\Phi$ that satisfies \eqref{230114e3_15}, and have arrived a new phase function $\Phi_{s_{d_{1}}}$ that satisfies a ``better" derivative bound.

Consider the amplitude function in \eqref{230211e2_66}. Recall that in this case, $\theta$ is supported in a small neighbourhood of the origin. Moreover, 
\begin{equation}
    (v_{d_{1}-2})^2+\dots +(v_2+\xi')^2\simeq 1.
\end{equation}
We can therefore cut the support of $\circv_{d_{1}-1}$ into $O_d(1)$ pieces, expressed by 
\begin{equation}
    \sum_{2\le d'_1\le d_1-2}b_{2, d'_1}(\circv_{d_1-1})=1,
\end{equation}
for every relevant $\circv_{d_1-1}$, so that on each piece it holds that 
\begin{equation}
    \anorm{
    \frac{
    \partial^{d_{1}'}}{\partial \theta^{d_{1}'}}
    \Phi_{s_{d_{1}}}(\theta; \circv_{d_{1}-1}; \bxi)
    }\gtrsim 1,
\end{equation}
for some $2\le d_{1}'\le d_{1}-2$.

\subsection{The second step of the algorithm}

Before describing the general steps of the algorithm, we still need to describe the second step, as there is more data we need to record in the algorithm. \\

Recall that after the first step, we have controlled the left hand side of the desired estimate \eqref{230209e2_8} by \eqref{230210e2_61}. The data there are given by \eqref{230210e2_63} and satisfy \eqref{230212e2_68}. Recall that the algorithm terminates if we are in the case $\mathfrak{d}_1=2$ or $\mathfrak{k}_1=0$. Otherwise, we proceed as follows.

Similarly to \eqref{230122e_18}--\eqref{221123e2_14}, for the phase function $\Phi_{\mathfrak{s}_1}(\theta; \circu_{\mathfrak{m}_1}; \bxi)$ on the support of the amplitude function $\mf{a}^{(1)}(\theta; \circu_{\mf{m}_1}; \bxi)$, we can cut the support of $\mf{a}^{(1)}$ in the $\circu_{\mf{m}_1}$ variable into $O_d(1)$ many pieces such that on each piece, either it holds that 
\begin{equation}\label{230212e2_74}
    \anorm{
    \frac{
    \partial^{d_2}
    }{
    \partial \theta^{d_2}
    }   
    \Phi_{\mathfrak{s}_1}(\theta; \circu_{\mathfrak{m}_1}; \bxi)
    }
    \gtrsim 1, \ \ \forall \theta,
\end{equation}
and 
\begin{equation}\label{230212e2_75}
    \frac{
    \partial^{d_2}
    }{
    \partial \theta^{d_2}
    }   
    \Phi_{\mathfrak{s}_1}(\theta; \circu_{\mathfrak{m}_1}; \bxi)=0
\end{equation}
admits a unique solution in the $\theta$ variable in a slightly enlarged interval of the $\theta$ support of $\mf{a}^{(1)}$ for some $2<d_2< \mf{d}_1$, or that \eqref{230212e2_74} holds for $d_2=2$. We proceed in essentially the same way as in the first step, with one difference that we explain now. \\

Let $\theta_{d_2-1}=\theta_{d_2-1}(\circu_{\mf{m}_1})$ be the unique solution to \eqref{230212e2_75}. Similarly to \eqref{230115e3_17}, we define 
$\Psi_{d_2-1, \iota}(
\circu_{
\mf{m}_1
}
)$ for $\iota=1, \dots, d$, and similarly to \eqref{221119e4_24} we make the change of variables 
\begin{equation}\label{230212e2_76}
    \Psi_{d_2-1, d}(
    \circu_{
    \mf{m}_1
    }
    )\mapsto v_d, \dots, \Psi_{d_2-1, d_2}(
    \circu_{
    \mf{m}_1
    }
    )\mapsto v_{d_2}, \Psi_{d_2-1, d_2-2}(
    \circu_{
    \mf{m}_1
    }
    )\mapsto v_{d_2-2}, \dots, \Psi_{d_2-1, 1}(
    \circu_{
    \mf{m}_1
    }
    )\mapsto v_1.
\end{equation}
The only difference is in Claim \ref{221124claim2_4}. In the current case, we have 
\begin{claim}\label{230212claim2_5}
    The change of variables in \eqref{230212e2_76} has a Jacobian comparable to $(\mf{s}_1)^{\frac{
    (d-\mathfrak{d}_1+1)(d-\mathfrak{d}_1)
    }{2}}$. Moreover, the range of $\circu_{
    \mf{m}_1
    }$, which is contained in $[-C, C]^{d-1}$ for some constant $C=C_\gamma$, will be transformed to a region contained in the rectangular box of dimensions
    \begin{equation}\label{230212e2_77}
        \Big( (
        \mathfrak{s}_1
        )^{d-
        \mathfrak{d}_1
        },  \dots,  (
        \mathfrak{s}_1
        )^{\mathfrak{d}_1-
        \mathfrak{d}_1
        },  1, \dots, 1\Big):=\mathfrak{D}(
        \mathfrak{d}_1, \mathfrak{s}_1
        )
    \end{equation}
    written in the order $(v_d, \dots, v_{d_2}, v_{d_2-2}, \dots, v_1)$. 
\end{claim}
\begin{proof}
    The proof of Claim \ref{230212claim2_5} is similar to that of Claim \ref{221124claim2_4}, and therefore we leave it out. 
\end{proof}

After Claim \ref{230212claim2_5}, everything else remains the same as in the first step. We write down directly the output of this step: We have 
\begin{equation}
    \mathfrak{d}_2, \ \ \mathfrak{m}_2, \ \ \mathfrak{s}_2, \ \ \mathfrak{k}_2, \ \ \Phi_{\mathfrak{s}_2}(\theta; \circu_{\mathfrak{m}_2}; \bxi),
\end{equation}
satisfying 
\begin{equation}
    \mathfrak{d}_2< \mathfrak{d}_1, \ \ 
    \mathfrak{s}_2\in 
    \begin{cases}
    \{1\}, & \text{ if  } \mathfrak{d}_2=2;\\
    [2^{-\frac{\mathfrak{k}_1}{\mathfrak{d}_2}} , \frac{1}{2}] & \text{ if } \mathfrak{d}_2>2, 
    \end{cases}
    \ \ \text{ and }  2^{\mathfrak{k}_2}=2^{\mathfrak{k}_1} (\mathfrak{s}_2)^{\mathfrak{d}_2}.
\end{equation}
The left hand side of \eqref{230209e2_8} is now controlled by 
\begin{equation}
\begin{split}
    & (\mf{s}_1)^{-\frac{
    (d-\mathfrak{d}_1+1)(d-\mathfrak{d}_1)
    }{2p}} \prod_{\iota=1}^2 \Big( 
     \mathfrak{s}_{\iota} (\mathfrak{s}_{\iota})^{
    \frac{(\mathfrak{m}_{\iota}+2)(\mathfrak{m}_{\iota}-1)}{2p}
    } 
    \jac(\mathfrak{s}_{\iota})
    \Big) \\
    &
    \times \Big( 
    \sum_{\Theta_{\mathfrak{d}_2}}
    \Norm{
    \iint_{\R^2}
    \widehat{f}_{\Theta_{\mathfrak{d}_2}}(\bxi) 
    \bnorm{
    \int_{\R} e^{i 2^{\mathfrak{k}_2}
    \Phi_{\mathfrak{s}_2}(\theta; \circu_{\mathfrak{m}_2}; \bxi)}
    \mathfrak{a}^{(2)}(\theta; \circu_{\mathfrak{m}_2}; \bxi)d\theta
    }
    e^{i(\bfx\cdot \bxi)}d\bxi
    }_{L^p
    }^p\Big)^{\frac{1}{p}}.
\end{split}
\end{equation}
where 
\begin{equation}
    \jac(\mathfrak{s}_{\iota}):=(\mathfrak{s}_{\iota})^{
    \mathfrak{d}_{\iota}-1
    }
    (\mathfrak{s}_{\iota})^{
    -\frac{\mathfrak{d}_{\iota}-1}{p}
    }
\end{equation}
and $\mathfrak{a}^{(2)}$ is a smooth amplitude function whose support in the $\circu_{\mathfrak{m}_2}$ variables is contained in the rectangular box of dimension given by \eqref{230212e2_77} and centered at the origin.

\subsection{Outputs of the whole algorithm}

We will keep running the above algorithm, until we reach the following stopping conditions. Assume that we have finished the $j$-th iteration of the algorithm. We will terminate the algorithm if $\mathfrak{d}_j=2$ or $\mathfrak{k}_j=0$. Otherwise we will run another step. Note that each time we run the algorithm, the derivative degree $\mathfrak{d}_j$ decreases by at least one, and therefore the algorithm terminates within $d$ steps. \\

Let us assume that the algorithm terminates after the $J$-th step. We arrive at 
\begin{equation}\label{230212e2_82}
\begin{split}
    & \prod_{j=1}^{J-1} \pnorm{
    (\mf{s}_j)^{-\frac{
    (d-\mathfrak{d}_j+1)(d-\mathfrak{d}_j)
    }{2p}}
    }
    \prod_{j=1}^J \Big( 
     \mathfrak{s}_{j} (\mathfrak{s}_{j})^{
    \frac{(\mathfrak{m}_{j}+2)(\mathfrak{m}_{j}-1)}{2p}
    } 
    \jac(\mathfrak{s}_{j})
    \Big) \\
    &
    \times \Big( 
    \sum_{\Theta_{\mathfrak{d}_J}}
    \Norm{
    \iint_{\R^2}
    \widehat{f}_{\Theta_{\mathfrak{d}_J}}(\bxi) 
    \bnorm{
    \int_{\R} e^{i 2^{\mathfrak{k}_J}
    \Phi(\theta; \circu_{\mathfrak{m}_J}; \bxi)}
    \mathfrak{a}^{(J)}(\theta; \circu_{\mathfrak{m}_J}; \bxi)d\theta
    }
    e^{i(\bfx\cdot \bxi)}d\bxi
    }_{L^p
    }^p\Big)^{\frac{1}{p}}.
\end{split}
\end{equation}
Here we have done a change of variables to turn the phase function $\Phi_{
\mathfrak{s}_J
}$ to $\Phi$, and 
the support of $a^{(J)}$ in the $\circu_{
\mathfrak{m}_J
}$ variables is contained in the rectangular box of dimensions given by 
\begin{equation}
    \prod_{j=1}^{J-1} \mathfrak{D}(
    \mathfrak{d}_j, \mathfrak{s}_j
    )
\end{equation}
where the above product is defined component-wise. It remains to prove that \eqref{230212e2_82} can be controlled by the right hand side of  \eqref{230209e2_8}.

\subsection{Final reduction}

After the termination of the above algorithm, one could proceed to prove Proposition \ref{230329prop2_2} directly. However, for some technical reasons, we will carry out a further reduction. The goal of this section is to make a further reduction to the phase function in \eqref{230212e2_82} so that its missing degree becomes $1$. Recall that the algorithm terminates after $J$-th step. We therefore either have $\mathfrak{d}_J=2$ or $\mathfrak{k}_J=0$. In the latter case, the oscillatory integral in \eqref{230212e2_82} does not oscillate anymore, and one can bound \eqref{230212e2_82} directly via standard argument. We therefore assume that we are in the case $\mathfrak{d}_J=2$. \\

If we are in the case $m_{J}=1$, then we do not do anything in this step. Let us assume that we are in the case $m_{J}>1$. Write the phase function $\Phi$ in \eqref{230212e2_82} as
\begin{equation}
    \Phi(\theta; \circu_{\mathfrak{m}_J}; \bxi)=\theta\xi+\eta P(\theta; \circu_{\mathfrak{m}_J}).
\end{equation}
Note that we have
\begin{equation}
    |\frac{\partial^2\Phi}{\partial\theta^2}|\gtrsim 1 \ \ \text{ on } \mathrm{supp}_1(\mathfrak{a}^{(J)}).
\end{equation}
Then there exists some $c_0=c_0(\gamma)$ such that 
\begin{equation}
    |\frac{\partial^2\Phi}{\partial\theta^2}|\gtrsim 1 \ \ \text{ on } (1+c_0)\mathrm{supp}_1(\mathfrak{a}^{(J)}).
\end{equation}
We cut the $\theta$-support of $\mathfrak{a}^{(J)}$ into $O(1)$ many pieces, each of which is of length $c_1\ll c_0$. Denote by $a_{c_1}$ the corresponding cutoff functions adapted to each such smaller interval. We also cut the $\xi'$-support of $\mathfrak{a}^{(J)}$ into $O(1)$ many pieces $\Theta_0$ with length $c_2$, where $c_2$ is to be determined.\\

Fix a $\Theta_0$ and write our phase function $\Phi$ as $\eta(\theta(\xi'-\xi'_0)+\theta\xi'_0+P(\theta; \circu_{\mathfrak{m}_J}))$, where $\xi'_0$ is the centre of $\Theta_0$. We have two cases. In the first case, the first order derivative 
\begin{equation}\label{230527e4_97}
    \frac{\partial}{\partial \theta}(\theta\xi'_0+P(\theta; \circu_{\mathfrak{m}_J}))
\end{equation}
 does not admit any zero on $(1+c_0)\mathrm{supp}(a_{c_1})$.  In this case, we note that
 \begin{equation}
     |\frac{\partial}{\partial \theta}(\theta\xi'_0+P(\theta; \circu_{\mathfrak{m}_J})|\ge c_3
 \end{equation}
 for all $\theta$ in $(1+c_0)\mathrm{supp}(a_{c_1})$ and for some $c_3>0$. Then we pick $c_2\ll c_3$, so that the first order derivative $\partial_{\theta}\Phi$ never admits any zero. Integration by parts will give us the desired bounds.

Let us consider the more interesting case, namely, the first order derivative in \eqref{230527e4_97} has a (unique) zero on $(1+c_0)\mathrm{supp}(a_{c_1})$. Call the solution $\theta_1(\circu_{\mathfrak{m}_J})$. In this case, we do the change of variables $\theta\mapsto \theta+\theta_1(\circu_{\mathfrak{m}_J})$ and compute the new phase function: 
\begin{equation}\label{Y_230520_e1119}
    \Phi(\theta+\theta_1(\circu_{\mathfrak{m}_J}); \circu_{\mathfrak{m}_J}; \bxi)=\Phi(\theta_1(\circu_{\mathfrak{m}_J}); \circu_{\mathfrak{m}_J}; \bxi)+(\xi-\eta\xi'_0)\theta+\sum_{j=2}^d 
    \frac{\partial^j }{\partial \theta^j}\Phi(\theta_1(\circu_{\mathfrak{m}_J}); \circu_{\mathfrak{m}_J}; \bxi)\frac{\theta^j}{j!}+E(\theta; \circu_{\mathfrak{m}_J}; \bxi), 
\end{equation}
where
\begin{equation}
    |E(\theta; \bfu_{
    \mathfrak{m}_J
    }; \bxi)|=O(|\theta|^{d+1}).
\end{equation}
We then do the following change of variables:
\begin{equation}
    x\mapsto x-\theta_1(\circu_{\mathfrak{m}_J}) ,\ y\mapsto y-\eta P(\theta_1(\circu_{\mathfrak{m}_J}); \circu_{\mathfrak{m}_J})
\end{equation}
so that we may assume without loss of generality that the first term of \eqref{Y_230520_e1119} vanishes. Then we do another change of variables
\begin{equation}\label{Y_230520_e1149}
    \frac{\partial^j }{\partial \theta^j} P(\theta_1(\circu_{\mathfrak{m}_J}); \circu_{\mathfrak{m}_J}) \mapsto v_j
\end{equation}
for all $j=2,3,\dots,d$.\begin{claim}\label{Y_230520_claim}
  The Jacobian of the change of variables in \eqref{Y_230520_e1149} is comparable to $1$.
\end{claim}
\begin{proof}[Proof of Claim \ref{Y_230520_claim}]
  We define
\begin{equation}
    \Psi_{1,\iota}(\circu_{\mathfrak{m}_J}):=(\frac{\partial^{\iota}}{\partial\theta^{\iota}}\Phi)(\theta_1(\circu_{\mathfrak{m}_J});\circu_{\mathfrak{m}_J};(\xi,1)),
\end{equation}
for $2\leq \iota \leq d+1$, and 
\begin{equation}
    \Psi_{1,1}(\circu_{\mathfrak{m}_J}):=(\frac{\partial\Phi}{\partial\theta})(\theta_1(\circu_{\mathfrak{m}_J});\circu_{\mathfrak{m}_J};(\xi,1))-\xi.
\end{equation}
Similar to the proof of Claim \ref{221124claim2_4}, we compute directly
\begin{equation}
    \partial_{u_i}v_j=\partial_{u_i}\Psi_{1,j}=\Psi_{1,j+1}\partial_{u_i}\theta_1+\partial_{u_i}\partial_{\theta}^i P(\theta_1(\circu_{\mathfrak{m}_J}); \circu_{\mathfrak{m}_J}),
\end{equation}
for all $i\ne \mathfrak m_J$, $2\le j\le d$. Taking a derivative in $u_i$ on both sides of the equation
\begin{equation}
    \xi'_0+\frac{\partial}{\partial\theta}P(\theta_1(\circu_{\mathfrak{m}_J}); \circu_{\mathfrak{m}_J})=0,
\end{equation}
we get
\begin{equation}
    \partial_{u_i}\theta_1=-\frac{\partial_{u_i}\partial_{\theta}P(\theta_1(\circu_{\mathfrak{m}_J}); \circu_{\mathfrak{m}_J})}{\Psi_{1,2}}.
\end{equation}
We consider the following $d\times d$ determinant:
\begin{equation}    \det\begin{bmatrix}
        \Psi_{1,2} & \Psi_{1,3} & \dots & \Psi_{1,d+1}\\
        \partial_{u_1}
        \partial_{\theta}
        P(\theta_1(\circu_{\mathfrak{m}_J}); \circu_{\mathfrak{m}_J})
         & \partial_{u_1}
         \partial_{\theta}^2
         P(\theta_1(\circu_{\mathfrak{m}_J}); \circu_{\mathfrak{m}_J})
         &\dots & \partial_{u_1}
         \partial_{\theta}^d
         P(\theta_1(\circu_{\mathfrak{m}_J}); \circu_{\mathfrak{m}_J})   \\
        \partial_{u_{\mathfrak m_J-1}}
        \partial_{\theta}
        P(\theta_1(\circu_{\mathfrak{m}_J}); \circu_{\mathfrak{m}_J})&\partial_{u_{\mathfrak m_J-1}}
        \partial_{\theta}^2
        P(\theta_1(\circu_{\mathfrak{m}_J}); \circu_{\mathfrak{m}_J})
        &\dots & \partial_{u_{\mathfrak m_J-1}}
        \partial_{\theta}^d
        P(\theta_1(\circu_{\mathfrak{m}_J}); \circu_{\mathfrak{m}_J})\\
        \vdots & \vdots & \ddots & \vdots \\
        \partial_{u_{\mathfrak m_J+1}}
        \partial_{\theta}
        P(\theta_1(\circu_{\mathfrak{m}_J}); \circu_{\mathfrak{m}_J})&\partial_{u_{\mathfrak m_J+1}}
        \partial_{\theta}^2
        P(\theta_1(\circu_{\mathfrak{m}_J}); \circu_{\mathfrak{m}_J})
        &\dots & \partial_{u_{\mathfrak m_J+1}}
        \partial_{\theta}^d
        P(\theta_1(\circu_{\mathfrak{m}_J}); \circu_{\mathfrak{m}_J})\\
        \vdots & \vdots & \ddots & \vdots \\
        \partial_{u_d}
        \partial_{\theta}
        P(\theta_1(\circu_{\mathfrak{m}_J}); \circu_{\mathfrak{m}_J}) & \partial_{u_d}
        \partial_{\theta}^2
        P(\theta_1(\circu_{\mathfrak{m}_J}); \circu_{\mathfrak{m}_J}) & \dots & \partial_{u_d}
        \partial_{\theta}^d
        P(\theta_1(\circu_{\mathfrak{m}_J}); \circu_{\mathfrak{m}_J}).        
\end{bmatrix}
\end{equation}
On the one hand, this is non-singular because of curvature condition \eqref{Y_230330nondegenerate}. On the other hand, one can check this is just comparable to the Jacobian we need to compute. We leave out the details since this is similar to the proof of Claim \ref{221124claim2_4}.
\end{proof}

After this change of variables, we now assume without loss of generality that the missing degree of the phase function $\Phi$ is $1$. That is, 
\begin{equation}
\Phi(\theta;\bfv;\bxi)=\eta \gamma(\theta;\bfv)+\theta \xi,
\end{equation}
where $\gamma(\theta;\bfv)=\sum_{j=2}^{d}v_j\frac{\theta^j}{j!}+O(|\theta|^{d+1})$ and $|v_2|\simeq 1$.


\section{Proof of Proposition \ref{230329prop2_2}}
Recall that our goal is to prove Proposition \ref{230329prop2_2}. We will see that picking $p_d=d(d+1)$  is more than enough. So far we have reduced it to estimating \eqref{230212e2_82}. That the above algorithm terminates after the $J$-th step means we are either in the case $\mathfrak{k}_J=0$ or $\mathfrak{d}_J=2$. These two cases will be handled differently. \\

Let us first work with the case $\mathfrak{k}_J=0$. In this case we have the trivial estimate
\begin{equation}
    \Norm{
    \iint_{\R^2}
    \widehat{f}_{\Theta_{\mathfrak{d}_J}}(\bxi)
    \bnorm{
    \int_{\R} e^{i
    \Phi(\theta; \bfv; \bxi)}
    \mathfrak{a}^{(J)}(\theta; \bfv; \bxi)d\theta
    }
    e^{i(\bfx\cdot \bxi)}d\bxi
    }_{L^p
    }\lesim \Norm{f_{\Theta_{\mathfrak{d}_J}}}_p.
\end{equation}
As a consequence, 
\begin{equation}
    \eqref{230212e2_82} \lesim \prod_{j=1}^{J-1} \pnorm{
    (\mf{s}_j)^{-\frac{
    (d-\mathfrak{d}_j+1)(d-\mathfrak{d}_j)
    }{2p}}
    }\prod_{j=1}^J \Big( 
     \mathfrak{s}_{j} (\mathfrak{s}_{j})^{
    \frac{(\mathfrak{m}_{j}+2)(\mathfrak{m}_{j}-1)}{2p}
    } 
    \jac(\mathfrak{s}_{j})
    \Big) 
    \pnorm{
    \sum_{
    \Theta_{\mathfrak{d}_J}
    } \norm{
    f_{
    \Theta_{\mathfrak{d}_J}}
    }_p^p
    }^{1/p}, 
\end{equation}
where we recall the definitions in \eqref{Y_230520} and \eqref{230208e2_51}. By reverting the changes of variables in \eqref{230411e3_63} and by interpolation between $L^2$ and $L^{\infty}$, we obtain 
\begin{equation}
    \pnorm{
    \sum_{
    \Theta_{\mathfrak{d}_J}
    } \norm{
    f_{
    \Theta_{\mathfrak{d}_J}}
    }_p^p
    }^{1/p} 
    \lesim 
    \norm{f}_p
    \prod_{j=1}^J (\jac(
    \mathfrak{s}_j
    ))^{-1} , 
\end{equation}
for every $p\ge 2$. It remains to prove 
\begin{equation}
    \prod_{j=1}^{J-1} \pnorm{
    (\mf{s}_j)^{-\frac{
    (d-\mathfrak{d}_j+1)(d-\mathfrak{d}_j)
    }{2p}}
    }\prod_{j=1}^J \Big( 
     \mathfrak{s}_{j} (\mathfrak{s}_{j})^{
    \frac{(\mathfrak{m}_{j}+2)(\mathfrak{m}_{j}-1)}{2p}
    } 
    \Big) \lesim 2^{-\frac{dk}{p}}.
\end{equation}
Note that in this case, $2^{-k}= \prod_{j=1}^J(\mathfrak{s}_j)^{\mathfrak{d}_j}$, thus 
\begin{equation}  \prod_{j=1}^J(\mathfrak{s}_j)^{d\mathfrak{d}_j}= 2^{-dk}, 
\end{equation}
and it suffices to prove
\begin{equation}
     p+\frac{(m_j+2)(m_j-1)-(d-\mathfrak{d}_j+1)(d-\mathfrak{d}_j)}{2} \ge d\mathfrak{d}_j
\end{equation}
for every $p\ge d(d+1)$.
We only need to show
\begin{equation}
    p-\frac{d(d+1)}{2}\ge \frac{\mathfrak{d}_j^2}{2}-\frac{\mathfrak{d}_j}{2}.
\end{equation}
But this is true since $p\ge d(d+1)$ and $\mathfrak{d}_j\leq d+1$.\\

It remains to handle the case $\mathfrak{d}_J=2$.  Let $\Phi(\theta;\bfv; \bxi)=\theta\xi+\eta \gamma(\theta; \bfv)$. Let $\theta_1(\bfv; \bxi)$ be the unique solution to 
\begin{equation}
    \frac{\partial}{\partial \theta} \Phi(\theta; \bfv
    ; \bxi)=0
\end{equation}
in the $\theta$ variable. Define 
\begin{equation}
    \Psi(\bfv
    ; \bxi):=\Phi(\theta_1(\bfv
    ; \bxi); \bfv; \bxi). 
\end{equation}
By the stationary phase principle, 
\begin{align}
    \eqref{230212e2_82} &\lesim 2^{
    -\frac{\mathfrak{k}_J}{2}
    }
    \prod_{j=1}^{J-1} \pnorm{
    (\mf{s}_j)^{-\frac{
    (d-\mathfrak{d}_j+1)(d-\mathfrak{d}_j)
    }{2p}}
    }    \prod_{j=1}^J \Big( 
     \mathfrak{s}_{j} (\mathfrak{s}_{j})^{
    \frac{(\mathfrak{m}_{j}+2)(\mathfrak{m}_{j}-1)}{2p}
    } 
    \jac(\mathfrak{s}_{j})
    \Big)\\
    &\times 
    \Big( 
    \sum_{\Theta_{\mathfrak{d}_J}}
    \Norm{
    \iint_{\R^2}
    \widehat{f}_{\Theta_{\mathfrak{d}_J}}(\bxi) 
    e^{i 2^{\mathfrak{k}_J}
    \Psi(\bfv; \bxi)}
    e^{i(\bfx\cdot \bxi)}d\bxi
    }_{L^p
    }^p\Big)^{\frac{1}{p}}.\label{Y_230526}
\end{align}

\begin{claim}\label{221126claim3_1}
We have the following curvature condition:  
\begin{equation}
    \det\begin{bmatrix}
    \partial_{v_2}\partial^2_{\xi}\Psi & \dots & \partial_{v_2} \partial_{\xi}^{d}\Psi\\
    \vdots & \ddots & \vdots\\
    \partial_{v_d}\partial^2_{\xi}\Psi & \dots & \partial_{v_d} \partial_{\xi}^{d}\Psi
    \end{bmatrix}(\bfv; \bxi)\neq 0,
\end{equation}
for any $\bfv, \bxi$ under consideration. 
\end{claim}
\begin{proof}[Proof of Claim \ref{221126claim3_1}] 
Without loss of generality, we assume $\eta=1$. Note that 
\begin{equation}
\partial_{\xi}\Psi(\bfv;\bxi)=\theta_1(\bfv;\bxi),   
\end{equation}
and
\begin{equation}
\partial_{v_i}\Psi(\bfv;\bxi)=\partial_{v_i}\gamma(\theta_1; \bfv).
\end{equation}
Taking a derivative in $\xi$ on both sides of 
\begin{equation}
    \xi+\partial_\theta\gamma(\theta_1; \bfv)=0,
\end{equation}
we get 
\begin{equation}
    \partial_{\xi}\theta_1=-\frac{1}{\partial^2_{\theta}\gamma(\theta_1; \bfv)}.
\end{equation}
Similarly, we have 
\begin{equation}
    \partial_{v_i}\theta_1=\partial_{v_i}\partial_{\theta}\gamma(\theta_1; \bfv)\partial_{\xi}\theta_1.
\end{equation}
So the required curvature condition is equal to:
\begin{align}
    &\det\begin{bmatrix}
        \partial_{v_2}\partial_{\xi}\theta_1(\bfv;\bxi) & \dots & \partial_{v_2}\partial_{\xi}^{d-1}
        \theta_1(\bfv;\bxi)\\
        \vdots & \ddots & \vdots\\
        \partial_{v_d}\partial_{\xi}\theta_1(\bfv;\bxi) & \dots & \partial_{v_d}\partial_{\xi}^{d-1}\theta_1(\bfv;\bxi)\\
    \end{bmatrix}\\
    &=\det\begin{bmatrix}
        \partial_{\xi}(\partial_{v_2}
        \partial_{\theta}
        \gamma(\theta_1; \bfv)\partial_{\xi}\theta_1)  & \dots & \partial_{\xi}^{d-1}(\partial_{v_2}
        \partial_{\theta}
        \gamma(\theta_1; \bfv)\partial_{\xi}\theta_1)\\
        \vdots & \ddots & \vdots \\
        \partial_{\xi}(\partial_{v_d} \partial_{\theta}
        \gamma(\theta_1; \bfv)\partial_{\xi}\theta_1) & \dots & \partial_{\xi}^{d-1}(\partial_{v_d}
        \partial_{\theta}\gamma(\theta_1; \bfv)\partial_{\xi}\theta_1)\\        \end{bmatrix}.
\end{align}
Let us consider the $d\times d$ determinant:
\begin{equation}\label{230527e5_21}
    \det\begin{bmatrix}
        \partial_{\theta}^2
        \gamma(\theta_1; \bfv)\partial_{\xi}\theta_1 & \partial_{\xi}(\partial_{\theta}^2
        \gamma(\theta_1; \bfv)\partial_{\xi}\theta_1) & \dots & \partial_{\xi}^{d-1}(\partial_{\theta}^2
        \gamma(\theta_1; \bfv)\partial_{\xi}\theta_1)\\
        \partial_{v_2}
        \partial_{\theta}
        \gamma(\theta_1; \bfv)\partial_{\xi}\theta_1 & \partial_{\xi}(\partial_{v_2}
        \partial_{\theta}
        \gamma(\theta_1; \bfv)\partial_{\xi}\theta_1)  & \dots & \partial_{\xi}^{d-1}(\partial_{v_2}
        \partial_{\theta}\gamma(\theta_1; \bfv)\partial_{\xi}\theta_1)\\
        \vdots & \vdots & \ddots & \vdots \\
        \partial_{v_d}
        \partial_{\theta}
        \gamma(\theta_1; \bfv)\partial_{\xi}\theta_1 & \partial_{\xi}(\partial_{v_d}
        \partial_{\theta}
        \gamma(\theta_1; \bfv)\partial_{\xi}\theta_1) & \dots & \partial_{\xi}^{d-1}(\partial_{v_d}
        \partial_{\theta}\gamma(\theta_1; \bfv)\partial_{\xi}\theta_1)\\        
        \end{bmatrix}.
    \end{equation}
On the one hand, by elementary row operations, we can use the first row to make the second row only have $\partial_{v_i}
\partial_{\theta}^2
\gamma(\theta_1; \bfv)(\partial_{\xi}\theta_1)^2$. Repeating this trick, we have this determinant is equal to
\begin{equation}
    \det\begin{bmatrix}
        \partial_{\theta}^2
        \gamma(\theta_1; \bfv) & \partial_{\theta}^3\gamma(\theta_1; \bfv) & \dots & \partial_{\theta}^{d+1}\gamma(\theta_1; \bfv)\\
        \partial_{v_2}\partial_{\theta}
        \gamma(\theta_1; \bfv) & \partial_{v_2}\partial^2_{\theta}
        \gamma(\theta_1; \bfv) & \dots & \partial_{v_2}
        \partial_{\theta}^d
        \gamma(\theta_1; \bfv)\\
        \vdots & \vdots & \ddots & \vdots\\
        \partial_{v_d}
        \partial_{\theta}
        \gamma(\theta_1; \bfv) & \partial_{v_d}
        \partial_{\theta}^2
        \gamma(\theta_1; \bfv) & \dots & \partial_{v_d}
        \partial_{\theta}^d
        \gamma(\theta_1; \bfv)\\        \end{bmatrix}(\partial_{\xi}\theta_1)^{\frac{(d+1)d}{2}}.
\end{equation}
On the other hand, using
\begin{equation}
    1+\partial_{\theta}^2\gamma(\theta_1; \bfv)\partial_{\xi}\theta_1=0
\end{equation}
we see the entries in the first row in \eqref{230527e5_21} is $0$ except the first term.  
 Since we are in the case $d_{1}=2$, this means $|\frac{\partial^2}{\partial\theta^2}\Phi|\gtrsim 1$, thus we have $|\partial_{\xi}\theta_1|\simeq 1$. Thus by the non-degeneracy condition \eqref{Y_230330nondegenerate}, we have finished the proof.

\end{proof}

After verifying the curvature condition in Claim \ref{221126claim3_1}, we are ready to prove the following decoupling estimate. Let $\Theta\subset \R$ be an interval of length $2^{-k_0/d}$. Let $\mathfrak{b}_{\Theta}$ be a frequency projection adapted to the interval $\Theta$. Denote 
\begin{equation}
    \widehat{f}_{\Theta}(\bxi):=\widehat{f}(\bxi)\cdot \mathfrak{b}_{\Theta}(\xi'). 
\end{equation}
Note this notation is different from \eqref{Y_230520} since we have no rescaling here.
\begin{lemma}\label{221126lemma3_2}
For every $\epsilon>0$ and $p\geq d(d+1)$, it holds that 
\begin{equation}
\begin{split}
& \Norm{
\iint_{\R^2} \widehat{f}(\bxi)
e^{i2^{k_0} \Psi(\bfv; \bxi)}
 e^{i\bfx\cdot \bxi} d\bxi 
}_{L^p_{\bfx; \bfv}(\R^2\times \R^{d-1})}\\ 
    & \lesim_{p, d, \epsilon} 2^{\epsilon k_0} 2^{\frac{k_0}{d}(1-\frac{1}{p}-\frac{d(d+1)}{2p})}
\Big(\sum_{\ell(\Theta)=2^{-k_0/d}}\Norm{
\iint_{\R^2} \widehat{f}_{\Theta}(\bxi)
e^{i2^{k_0} \Psi(\bfv; \bxi)}
 e^{i\bfx\cdot \bxi} d\bxi
}_{L^p_{\bfx; \bfv}(\R^2\times \R^{d-1})}^p \Big)^{\frac{1}{p}}.
\end{split}
\end{equation}
\end{lemma}
\begin{proof}[Proof of Lemma \ref{221126lemma3_2}]
The case $d=2$ is a special case of Theorem 1.4 of \cite{BHS20}. The general case $d\ge 3$ can be proven similarly, by combining the bootstrapping argument in Pramanik and Seeger \cite{PS07} and the decoupling inequalities of Bourgain, Demeter and Guth \cite{BDG16}. 
\end{proof}

Continuing the computation in \eqref{Y_230526}, we need to estimate the term

\begin{equation}
    2^{-\frac{\mathfrak{k}_J}{2}}\Norm{
    \iint_{\R^2}
    \widehat{f}_{\Theta_{\mathfrak{d}_J}}(\bxi)
    e^{i(\bfx\cdot \bxi+2^{\mathfrak{k}_J}
    \Psi(\bfv; \bxi))} 
    \mathfrak{a}^{(J)}(\theta_1; \bfv; \bxi)
    d\bxi
    }_{L^p
    (\R^2\times \prod_{j=1}^J \mathfrak{D}(
    \mathfrak{d}_j, \mathfrak{s}_j
    ))}.
\end{equation}
Denote the term in the integral by $T_{\Theta_{\mathfrak{d}_J}}f$, and write 
\begin{equation}
    2^{J_i}=2^{\mathfrak{k}_J} \prod_{j:\mathfrak{d}_j\leq i}(s_j)^{i-\mathfrak{d}_j},
\end{equation}
and 
$D_d=\frac{d(d+1)}{2}+1$. We consider two cases $J_d>0$ and $J_d\le 0$ separately. The former case is more interesting. 

Let us assume that $J_d>0$. Then by Lemma \ref{221126lemma3_2}, we have
\begin{equation}
    \Norm{T_{\Theta_{\mathfrak{d}_J}}f}_{L^p}\lesim 2^{\frac{J_d}{d}(1-\frac{D_d}{p})}\Big(\sum_{\ell(\Theta_d)=2^{-\frac{J_d}{d}}} \Norm{T_{\Theta_d}f}_p^p\Big)^{\frac{1}{p}}.
\end{equation}
for all $p\ge d(d+1)$.  Now for each $\Theta_d$ with $\ell(\Theta_d)=2^{-\frac{J_d}{d}}$, we  do not see the curvature in the last variable and hope to use Lemma \ref{221126lemma3_2} again, with   $d$ replaced by $d-1$. To this end, we need to check the following determinant condition:
\begin{equation}
    \det\begin{bmatrix}
    \partial_{v_2}\partial^2_{\xi}\Psi(\bfv; \bxi) & \dots & \partial_{v_2} \partial_{\xi}^{d-1}\Psi(\bfv; \bxi)\\
    \vdots & \ddots & \vdots\\
    \partial_{v_{d-1}}\partial^2_{\xi}\Psi(\bfv; \bxi) & \dots & \partial_{v_{d-1}} \partial_{\xi}^{d-1}\Psi(\bfv; \bxi)\\
    \end{bmatrix}\neq 0,
\end{equation}
for all $(\bfv,\bxi)$ under consideration. One can compute the curvature condition of $\gamma(\theta;\bfv)$ for the first $d-2$ parameters at $0$, which is comparable to $v_2$ and we know that $|v_2|\simeq 1$. Thus we get
\begin{equation}
    \Norm{T_{\Theta_d}f}_p\lesim 2^{(\frac{J_{d-1}}{d-1}-\frac{J_d}{d})(1-\frac{D_{d-1}}{p})}\Big(\sum_{\ell(\Theta_{d-1})=2^{-\frac{J_{d-1}}{d-1}}} \Norm{T_{\Theta_{d-1}}f}_p^p\Big)^{\frac{1}{p}}.
\end{equation}
Continue this process, we finally obtain 
\begin{equation}
    \Norm{T_{\Theta_{\mathfrak{d}_J}}f}_{L^p}\lesim 2^{\frac{J_d}{d}(1-\frac{D_d}{p})}\cdot 2^{(\frac{J_{d-1}}{d-1}-\frac{J_d}{d})(1-\frac{D_{d-1}}{p})}\cdots 2^{(\frac{J_2}{2}-\frac{J_3}{3})(1-\frac{D_2}{p})} \Big(\sum_{\ell(\Theta_{2})=2^{-\frac{J_2}{2}}} \Norm{T_{\Theta_2}f}_p^p \Big)^{\frac{1}{p}}.
\end{equation}
After reaching this scale, there is no essential oscillation in the integral. By Young's inequality, we have 
\begin{equation}
    \Norm{T_{\Theta_2}f}_p\lesim \Norm{f_{\Theta_2}}_p.
\end{equation}
Putting everything together, we need to show
\begin{equation}
\prod_{j=1}^{J-1} \pnorm{
    (\mf{s}_j)^{-\frac{
    (d-\mathfrak{d}_j+1)(d-\mathfrak{d}_j)
    }{2p}}
    }  \prod_{j=1}^J \Big( 
     \mathfrak{s}_{j} (\mathfrak{s}_{j})^{
    \frac{(\mathfrak{m}_{j}+2)(\mathfrak{m}_{j}-1)}{2p}
    } 
    \Big) 2^{-\frac{J_2}{p}}\prod_{j=2}^d 2^{-\frac{J_j}{p}} \lesim 2^{-\frac{dk}{p}}.
\end{equation}
Recall 
\begin{equation}
2^{J_i}=2^{\mathfrak{k}_J}\prod_{j:\mathfrak{d}_j\leq i}(\mathfrak{s}_j)^{i-\mathfrak{d}_j}=2^{k}\prod_{j=1}^{J}(\mathfrak{s}_j)^{\mathfrak{d}_j}\prod_{j:\mathfrak{d}_j\leq i}(\mathfrak{s}_j)^{i-\mathfrak{d}_j}.
\end{equation}
Substituting back to the previous term, it suffices to show 
\begin{equation}
    \prod_{j=1}^J (\mathfrak{s}_j)^{p-\frac{(d-\mathfrak{d}_j+1)(d-\mathfrak{d}_j)}{2}+\frac{(m_j+2)(m_j-1)}{2}-\mathfrak{d}_j-\sum_{i:i\ge \mathfrak{d_j}}i-\sum_{i:i<\mathfrak{d}_j}\mathfrak{d}_j} \lesim 1,
\end{equation}
which is equivalent to showing
\begin{equation}
    p-\frac{(d-\mathfrak{d}_j+1)(d-\mathfrak{d}_j)}{2}+\frac{(m_j+2)(m_j-1)}{2}-\mathfrak{d}_j-\sum_{i:i\ge \mathfrak{d_j}}i-\sum_{i:i<\mathfrak{d}_j}\mathfrak{d}_j \ge 0.
\end{equation}
Rewrite the left hand side as
\begin{equation}
    p-\frac{(d-\mathfrak{d}_j+1)(d-\mathfrak{d}_j)}{2}+\frac{(m_j+2)(m_j-1)}{2}-\mathfrak{d}_j-\sum_{i:i\ge \mathfrak{d_j}}(i-\mathfrak{d}_j)+\mathfrak{d}_j-\sum_{i:i<\mathfrak{d}_j}\mathfrak{d}_j.
\end{equation}
This is equal to
\begin{equation}
    p-(d-\mathfrak{d}_j+1)(d-\mathfrak{d}_j)-d\mathfrak{d}_j+\frac{(m_j+2)(m_j-1)}{2}\ge p-(d-\mathfrak{d}_j+1)(d-\mathfrak{d}_j)-d\mathfrak{d}_j.
\end{equation}
Using the fact that $p\ge d(d+1)$, we are lead to showing that
\begin{equation}
    (d+1)\mathfrak{d}_j-\mathfrak{d}_j^2\ge 0,
\end{equation}
which holds since $\mathfrak{d}_j\leq d+1$.\\

In the end, we consider the case $J_d\le 0$. If $J_d\leq 0$, we turn to $J_i$, where $i$ is the largest number such that $J_i> 0$ and use Lemma \ref{221126lemma3_2} with degree $i$. We apply the same decoupling procedure above. Note that the above case is the worse case when we apply decoupling. To see this, assume without loss of generality that $J_d\leq 0$ and $J_{d-1}>0$. When we decouple to the scale $2^{\frac{J_{d-1}}{d-1}}$, we get a factor $2^{\frac{J_{d-1}}{d-1}(1-\frac{D_{d-1}}{p})}$, while the above case gives us $2^{-\frac{J_d}{p}}\cdot 2^{\frac{J_{d-1}}{d-1}(1-\frac{D_{d-1}}{p})}$, which is bigger since $J_d\leq 0$. This finishes the proof of Proposition \ref{230329prop2_2}.

\section{Variable coefficient maximal operators}

In this section, we will prove Theorem \ref{231102theorem1_3}. Recall that in the proof of Theorem \ref{230329theorem1_1}, the cinematic curvature condition \ref{Y_230330nondegenerate} appears in the step of proving local smoothing estimates (more precisely, in Claim \ref{221124claim2_4} about the non-degeneracy of a change of variables and in Claim \ref{221126claim3_1} which allowed us to use Fourier decoupling inequalities).  

In the rest of this section, we will focus on these two differences, and leave out the rest of the proof. 

We would like to remark that reduction to the normal form is not essential to the proof. In order to apply our reduction algorithm, we only need
\begin{align}
    \partial_{\theta^{d'}}\gamma(0)  \neq 0 
\end{align}
for some $d'\in [2,d+1]$.

If for all $d’\in [2,d+1]$, $\partial_{\theta^{d'}}\gamma(0)=0$, then we consider the first column of the matrix in (H3). There must be a smallest $d'\in [2,d+1]$ such that $(\partial_{\theta^{d'}}+\partial_{x\theta^{d'-1}})\gamma(0)\ne 0$. Since we can do a nonlinear transform 
\begin{align}
    (x, y)\mapsto (x, y+ c x^{d'}),
\end{align}
without loss of generality, we can assume that there is a smallest $d'\in [2,d+1]$, such that $\partial_{\theta^{d'}}\gamma(0)\neq 0$.

\subsection{Reduction algorithm}
We follow the same reduction algorithm in the translation-invariant case. We only need to check Claim \ref{221124claim2_4}. Denote our phase function of the multiplier by
\begin{align}
    \Phi(\bfx;\theta;\bfu;\bxi):=\theta\xi+\gamma(x,y;\bfu;\theta)\eta
\end{align}
where $\bfu=(u_1,u_2,\cdots,u_{d-1})$. We assume 
$\theta_{d_1-1}=\theta_{d_1-1}(\bfx;\bfu)$ is the unique solution to 
\begin{align}
    \frac{\partial^{d_1-1}}{\partial \theta^{d_1-1}}\Phi(\bfx;\theta;\bfu;\bxi)=0
\end{align}
We do a change of variables 
\begin{align}
    \theta\mapsto \theta+\theta_{d_1-1}(\bfx;\bfu)
\end{align}
Then the new phase function is 
\begin{align}
    &\Phi(\bfx;\theta+\theta_{d_1-1}(\bfx;\bfu);\bfu;\bxi)\\
    &=\Phi(\bfx;\theta_{d_1-1}(\bfx;\bfu);\bfu;\bxi)+(\eta\Psi_{d_1-1,1}(\bfx;\bfu)+\xi)\theta+\eta\sum_{i=2}^d \Psi_{d_1-1,i}(\bfx;\bfu)\frac{\theta^i}{i!} +\eta\theta^{d+1}P(\bfx;\theta;\bfu;\bxi)
\end{align}
where $P$ is a smooth function and 
\begin{align}
    \Psi_{d_1-1,i}(\bfx;\bfu):=(\frac{\partial^i}{\partial \theta^i}\Phi)(\bfx;\theta_{d_1-1}(\bfx;\bfu);\bfu;(\xi,1))
\end{align}
for $2\leq i\leq d$, and 
\begin{align}
    \Psi_{d_1-1,1}(\bfx;\bfu):=\frac{\partial}{\partial \theta}\Phi(\bfx;\theta_{d_1-1}(\bfx;\bfu);\bfu;(\xi,1))-\xi
\end{align}
We do the change of variable
\begin{align}\label{var}
    x-\theta_{d_1-1}(\bfx;\bfu)\mapsto x,\ y-\gamma(\bfx;\bfu;\theta_{d_1-1}(\bfx;\bfu))\mapsto y,\ \Psi_{d_1-1,d}(\bfx;\bfu)\mapsto v_d, \cdots, \Psi_{d_1-1,1}(\bfx;\bfu)\mapsto v_1
\end{align}
This is a change of variable from $(\bfx,\bfu)$ to $(\bfx,\bfv)$, where $\bfv=(v_1,\cdots,v_{d_1-2},v_{d_1},\cdots,v_d)$.
\begin{claim}\label{Claim 7.3}
    The Jacobian of the change of variables in \eqref{var} has absolute value comparable to 1.
\end{claim}
\begin{proof}[Proof of Claim \eqref{Claim 7.3}]
Without loss of generality, we take $\eta=1$. We need to compute
\begin{align}
    \det\begin{bmatrix}
        1-\frac{\partial\theta_{d_1-1}}{\partial x} & -\frac{\partial\theta_{d_1-1}}{\partial y} & -\frac{\partial\theta_{d_1-1}}{\partial u_1} & \cdots & -\frac{\partial\theta_{d_1-1}}{\partial u_{d-1}}\\
         -\frac{\partial\gamma}{\partial x} & 1-\frac{\partial\gamma}{\partial y} & -\frac{\partial\gamma}{\partial u_1} & \cdots & -\frac{\partial\gamma}{\partial u_{d-1}}\\         
         \frac{\partial}{\partial x}(\frac{\partial}{\partial \theta}\gamma(\theta_{d_1-1})) &\frac{\partial}{\partial y}(\frac{\partial}{\partial \theta}\gamma(\theta_{d_1-1})) &\frac{\partial}{\partial u_1}(\frac{\partial}{\partial \theta}\gamma(\theta_{d_1-1})) & \cdots & \frac{\partial}{\partial u_{d-1}}(\frac{\partial}{\partial \theta}\gamma(\theta_{d_1-1}))  \\
         \cdots & \cdots & \cdots & \cdots & \cdots\\
\frac{\partial}{\partial x}(\frac{\partial^d}{\partial \theta^d}\gamma(\theta_{d_1-1})) &\frac{\partial}{\partial y}(\frac{\partial^d}{\partial \theta^d}\gamma(\theta_{d_1-1})) &\frac{\partial}{\partial u_1}(\frac{\partial^d}{\partial \theta^d}\gamma(\theta_{d_1-1})) & \cdots & \frac{\partial}{\partial u_{d-1}}(\frac{\partial^d}{\partial \theta^d}\gamma(\theta_{d_1-1}))          
         \end{bmatrix}    
\end{align}
Since 
\begin{align}
    \frac{\partial \gamma}{\partial x}(0)= \frac{\partial \gamma}{\partial y}(0)=0    
\end{align}
the determinant is comparable to 
\begin{align}\label{7.31}
    \det \begin{bmatrix}
    0 & -1 & \frac{\partial \gamma}{\partial u_1} & \cdots  & \frac{\partial \gamma}{\partial u_{d-1}}\\
    (1-\frac{\partial \theta_{d_1-1}}{\partial x}) & -\frac{\partial \theta_{d_1-1}}{\partial y} & -\frac{\partial\theta_{d_1-1}}{\partial u_1} & \cdots & -\frac{\partial\theta_{d_1-1}}{\partial u_{d-1}}\\
    \frac{\partial}{\partial x}(\frac{\partial}{\partial \theta}\gamma(\theta_{d_1-1})) & \frac{\partial}{\partial y}(\frac{\partial}{\partial \theta}\gamma(\theta_{d_1-1})) &   \frac{\partial}{\partial u_1}(\frac{\partial}{\partial \theta}\gamma(\theta_{d_1-1})) & \cdots & \frac{\partial}{\partial u_{d-1}}(\frac{\partial}{\partial \theta}\gamma(\theta_{d_1-1}))  \\    
     \cdots & \cdots & \cdots & \cdots\\
\frac{\partial}{\partial x}(\frac{\partial^d}{\partial \theta^d}\gamma(\theta_{d_1-1})) & \frac{\partial}{\partial y}(\frac{\partial^d}{\partial \theta^d}\gamma(\theta_{d_1-1})) & \frac{\partial}{\partial u_1}(\frac{\partial^d}{\partial \theta^d}\gamma(\theta_{d_1-1})) & \cdots & \frac{\partial}{\partial u_{d-1}}(\frac{\partial^d}{\partial \theta^d}\gamma(\theta_{d_1-1}))  
        \end{bmatrix}
\end{align}
Taking partial derivatives on both side of 
\begin{align}
    \frac{\partial^{d_1-1}}{\partial \theta^{d_1-1}}\Phi(\bfx;\theta_{d_1-1}(\bfx;\bfu);\bfu;\bxi)=0
\end{align}
we obtain
\begin{align}
    \frac{\partial \theta_{d_1-1}}{\partial x}=-\frac{\partial_{x\theta^{d_1-1}}\gamma}{\partial_{\theta^{d_1}}\gamma}\\
    \frac{\partial \theta_{d_1-1}}{\partial y}=-\frac{\partial_{y\theta^{d_1-1}}\gamma}{\partial_{\theta^{d_1}}\gamma}\\    \frac{\partial \theta_{d_1-1}}{\partial u_i}=-\frac{\partial_{u_i\theta^{d_1-1}}\gamma}{\partial_{\theta^{d_1}}\gamma}
\end{align}
Substituting into \eqref{7.31}, we have
\begin{align}
   \det \begin{bmatrix}
   0 & -1  & \cdots  & \frac{\partial \gamma}{\partial u_{d-1}}\\
   (1-\frac{\partial \theta_{d_1-1}}{\partial x}) & -\frac{\partial\theta_{d_1-1}}{\partial y} &  \cdots & -\frac{\partial\theta_{d_1-1}}{\partial u_{d-1}}\\
 \partial_{x\theta}\gamma+\partial_{\theta^2}\gamma \frac{\partial\theta_{d_1-1}}{\partial x} &     \partial_{y\theta}\gamma+\partial_{\theta^2}\gamma \frac{\partial\theta_{d_1-1}}{\partial y} &  
  \cdots & \partial_{u_{d-1}\theta}\gamma+\partial_{\theta^2}\gamma \frac{\partial\theta_{d_1-1}}{\partial u_{d-1}}  \\    
     \cdots & \cdots & \cdots & \cdots\\
\partial_{x\theta^d}\gamma+\partial_{\theta^{d+1}}\gamma \frac{\partial\theta_{d_1-1}}{\partial x} &     \partial_{y\theta^d}\gamma+\partial_{\theta^{d+1}}\gamma \frac{\partial\theta_{d_1-1}}{\partial y} & \cdots & \partial_{u_{d-1}\theta^d}\gamma+\partial_{\theta^{d+1}}\gamma \frac{\partial\theta_{d_1-1}}{\partial u_{d-1}} 
        \end{bmatrix}   
\end{align}
\begin{align}
   =(\partial_{\theta^{d_1}}\gamma) \det \begin{bmatrix}   
   0 & -1 & \cdots  & \frac{\partial \gamma}{\partial u_{d-1}}\\
   \partial_{\theta^{d_1}}\gamma+\partial_{x\theta^{d_1-1}}\gamma & \partial_{y\theta^{d_1-1}}\gamma & \cdots & \partial_{u_{d-1}\theta^{d_1-1}}\gamma\\  
\partial_{x\theta}\gamma+\partial_{\theta^2}\gamma & \partial_{y\theta}\gamma & \cdots & \partial_{u_{d-1}\theta}\gamma\\
\cdots & \cdots & \cdots & \cdots \\
\partial_{x\theta^d}\gamma+\partial_{\theta^{d+1}}\gamma & \partial_{y\theta^d}\gamma & \cdots & \partial_{u_{d-1}\theta^d}\gamma   \end{bmatrix}
\end{align}
The first term is comparable to 1, the second term is exactly our curvature condition (H3) up to a sign. This finishes the proof.
\end{proof}

\subsection{Decoupling condition}
In this section, we are going to check the decoupling condition which we need in the proof, i.e. the curvature condition for the variable coefficient decoupling inequality \eqref{221126claim3_1}.

Let us look at the curvature condition for the decoupling inequality. In the variable coefficient case, our phase function is
\begin{equation}
    x\xi+y\eta-\Phi(\bfx;\theta_1(\bfx;\bfv;\xi');\bfv;\bxi)
\end{equation}
where $\theta_1(\bfx;\bfv;\xi')$ is the unique solution of 
\begin{equation}
    \xi'+\partial_{\theta}\gamma(\bfx;\theta;\bfv)=0.
\end{equation}
Denote this phase function by $P$. The curvature condition we need here is
\begin{equation}\label{250127e8_22}
    \det\begin{bmatrix}
        \partial_y P & \partial_{y}\partial_{\xi}P & \dots & \partial_{y} \partial_{\xi}^{d}P\\        
        \partial_x P & \partial_{x}\partial_{\xi}P & \dots & \partial_{x} \partial_{\xi}^{d}P\\
    \partial_{v_1}P & \partial_{v_1}\partial_{\xi}P & \dots & \partial_{v_1} \partial_{\xi}^{d}P\\
     \vdots & \vdots & \ddots & \vdots\\
    
    \partial_{v_{d-1}}P & \partial_{v_{d-1}}\partial_{\xi}P  & \dots & \partial_{v_{d-1}} \partial_{\xi}^{d}P
    \end{bmatrix}(\bfv; \bxi)\neq 0.  
\end{equation}
If this holds, then the phase function can be approximated (in the spirit of Pramanik and Seeger \cite{PS07}) by  $(y,x,\bfv)\cdot(1,\xi',\xi'^2,\dots,\xi'^d)\eta$, and then we can apply  decoupling inequalities in Lemma \ref{221126lemma3_2}.

Without loss of generality, we take $\eta=1$. By similar computation before, one can get
\begin{equation}
    \partial_{\xi}P=x-\theta_1,
\end{equation}
thus the determinant in the curvature condition equals to (up to a sign)
\begin{equation}
    \det\begin{bmatrix}
       \gamma_y-1 & \frac{\partial \theta_1}{\partial y} & \frac{\partial}{\partial \xi}(\frac{\partial \theta_1}{\partial y}) & \cdots & \frac{\partial^{d-1}}{\partial \xi^{d-1}}(\frac{\partial \theta_1}{\partial y}) \\
        \gamma_x-\xi &  \frac{\partial \theta_1}{\partial x}-1 & \frac{\partial}{\partial \xi}(\frac{\partial \theta_1}{\partial x}) & \dots & \frac{\partial^{d-1}}{\partial \xi^{d-1}}(\frac{\partial \theta_1}{\partial x})\\

    \gamma_{v_1} & \frac{\partial \theta_1}{\partial v_1} & \frac{\partial}{\partial \xi}(\frac{\partial \theta_1}{\partial v_1}) & \dots & \frac{\partial^{d-1}}{\partial \xi^{d-1}}(\frac{\partial \theta_1}{\partial v_1})\\      
    
    \vdots & \vdots & \ddots & \ddots & \vdots\\    

     \gamma_{v_{d-1}} & \frac{\partial \theta_1}{\partial v_{d-1}} & \frac{\partial}{\partial \xi}(\frac{\partial \theta_1}{\partial v_{d-1}}) & \dots & \frac{\partial^{d-1}}{\partial \xi^{d-1}}(\frac{\partial \theta_1}{\partial v_{d-1}})\\    \end{bmatrix}
\end{equation}
Since 
\begin{align}
    \frac{\partial \theta_1}{\partial v_i}=\partial_{v_i\theta}\gamma\cdot \frac{\partial \theta_1}{\partial \xi}, \\
    \frac{\partial \theta_1}{\partial x}=\partial_{x\theta}\gamma \cdot \frac{\partial \theta_1}{\partial \xi},\\
    \frac{\partial \theta_1}{\partial y}=\partial_{y\theta}\gamma \cdot \frac{\partial \theta_1}{\partial \xi},
\end{align}
and
\begin{align}
    \frac{\partial \theta_1}{\partial \xi}=-\frac{1}{\gamma_{\theta\theta}},
\end{align}
the determinant is equal to
\begin{align}
   (\partial_{\xi}\theta_1)\det\begin{bmatrix}
       \gamma_y-1 & \gamma_{y\theta} & \frac{\partial}{\partial \xi}(\frac{\partial \theta_1}{\partial y}) & \cdots & \frac{\partial^{d-1}}{\partial \xi^{d-1}}(\frac{\partial \theta_1}{\partial y}) \\
        \gamma_x+\gamma_{\theta} &  \gamma_{x\theta}+\gamma_{\theta\theta} & \frac{\partial}{\partial \xi}(\frac{\partial \theta_1}{\partial x}) & \dots & \frac{\partial^{d-1}}{\partial \xi^{d-1}}(\frac{\partial \theta_1}{\partial x})\\
    \gamma_{v_1} & \gamma_{v_1\theta} & \frac{\partial}{\partial \xi}(\frac{\partial \theta_1}{\partial v_1}) & \dots & \frac{\partial^{d-1}}{\partial \xi^{d-1}}(\frac{\partial \theta_1}{\partial v_1})\\    
    \vdots & \vdots & \ddots & \ddots & \vdots\\   
     \gamma_{v_{d-1}} & \gamma_{v_{d-1}\theta} & \frac{\partial}{\partial \xi}(\frac{\partial \theta_1}{\partial v_{d-1}}) & \dots & \frac{\partial^{d-1}}{\partial \xi^{d-1}}(\frac{\partial \theta_1}{\partial v_{d-1}})\\    \end{bmatrix}
\end{align}
Since
\begin{align}
    &\frac{\partial}{\partial \xi}(\frac{\partial \theta_1}{\partial v_i}) =  \frac{\partial}{\partial \xi}(\gamma_{v_i\theta}\frac{\partial \theta_1}{\partial \xi} )  = \gamma_{v_i\theta \theta}(\frac{\partial \theta_1}{\partial \xi})^2+\gamma_{v_i\theta} \gamma_{\theta\theta\theta} (\frac{\partial \theta_1}{\partial \xi})^3 \\
    &\frac{\partial}{\partial \xi}(\frac{\partial \theta_1}{\partial y}) =   \gamma_{y\theta \theta}(\frac{\partial \theta_1}{\partial \xi})^2+\gamma_{y\theta} \gamma_{\theta\theta\theta} (\frac{\partial \theta_1}{\partial \xi})^3 \\
    &\frac{\partial}{\partial \xi}(\frac{\partial \theta_1}{\partial x}) =   \gamma_{x\theta \theta}(\frac{\partial \theta_1}{\partial \xi})^2+\gamma_{x\theta} \gamma_{\theta\theta\theta} (\frac{\partial \theta_1}{\partial \xi})^3,
\end{align}
using elementary row operations, this is equal to
\begin{equation}
    (\partial_{\xi}\theta_1)^3\det\begin{bmatrix}
       \gamma_y-1 & \gamma_{y\theta} & \gamma_{y\theta\theta} & \cdots & \frac{\partial^{d-1}}{\partial \xi^{d-1}}(\frac{\partial \theta_1}{\partial y}) \\
        \gamma_x+\gamma_{\theta} &  \gamma_{x\theta}+\gamma_{\theta\theta} &\gamma_{x\theta\theta}+\gamma_{\theta\theta\theta} & \dots & \frac{\partial^{d-1}}{\partial \xi^{d-1}}(\frac{\partial \theta_1}{\partial x})\\
    \gamma_{v_1} & \gamma_{v_1\theta} & \gamma_{v_1\theta\theta} & \dots & \frac{\partial^{d-1}}{\partial \xi^{d-1}}(\frac{\partial \theta_1}{\partial v_1})\\    
    \vdots & \vdots & \ddots & \ddots & \vdots\\   
     \gamma_{v_{d-1}} & \gamma_{v_{d-1}\theta} & \gamma_{v_{d-1}\theta\theta} & \dots & \frac{\partial^{d-1}}{\partial \xi^{d-1}}(\frac{\partial \theta_1}{\partial v_{d-1}})\\    \end{bmatrix}
\end{equation}
By applying the same trick, the determinant in  \eqref{250127e8_22} is equal to 
\begin{equation}
    (\partial_{\xi}\theta_1)^{\frac{d(d+1)}{2}}\det\begin{bmatrix}
       \gamma_y-1 & \gamma_{y\theta} & \gamma_{y\theta\theta} & \cdots & \gamma_{y\theta^d} \\
        \gamma_x+\gamma_{\theta} &  \gamma_{x\theta}+\gamma_{\theta\theta} &\gamma_{x\theta\theta}+\gamma_{\theta\theta\theta} & \dots & \gamma_{x\theta^d}+\gamma_{\theta^{d+1}}  \\
    \gamma_{v_1} & \gamma_{v_1\theta} & \gamma_{v_1\theta\theta} & \dots & \gamma_{v_1\theta^d} \\    
    \vdots & \vdots & \ddots & \ddots & \vdots\\   
     \gamma_{v_{d-1}} & \gamma_{v_{d-1}\theta} & \gamma_{v_{d-1}\theta\theta} & \dots & \gamma_{v_{d-1}\theta^{d}} \\    \end{bmatrix}
\end{equation}
The first factor is always comparable to 1, and thus the curvature condition \eqref{250127e8_22} follows from our assumption (H3).

\normalem

\noindent 
Mingfeng Chen\\
Department of Mathematics, University of Wisconsin-Madison, Madison, WI-53706, USA \\
Email address: mchen454@math.wisc.edu\\

\noindent Shaoming Guo\\
Chern Institute of Mathematics and LPMC, Nankai University, Tianjin, China\\
and \\
Department of Mathematics, University of Wisconsin-Madison, Madison, WI-53706, USA \\
Email address: shaomingguo2018@gmail.com\\

\noindent Tongou Yang\\
Department of Mathematics, UCLA, Los Angeles, CA, USA\\
and \\
Department of Mathematics, University of Wisconsin-Madison, Madison, WI-53706, USA \\
Email address: tongouyang@math.ucla.edu

\end{document}